\pdfoutput=1 

\documentclass{birkmult}
\usepackage{amsfonts}
\usepackage{amscd}
\usepackage{amssymb}
\usepackage{amsthm}
\usepackage{amsmath}
\usepackage{stmaryrd}
\usepackage{graphicx}
\usepackage{color}
\usepackage{verbatim}
\usepackage{mathrsfs}
\usepackage{dsfont}

\input prepictex
\input pictex
\input postpictex

 \theoremstyle{plain}
\newtheorem{thm}{Theorem}[section]
\newtheorem{lemma}[thm]{Lemma}
\newtheorem{prop}[thm]{Proposition}
\newtheorem{cor}[thm]{Corollary}
\theoremstyle{definition}
\newtheorem{example}[thm]{Example}
\newtheorem{defn}[thm]{Definition}
\newtheorem{remark}[thm]{Remark}
\theoremstyle{remark}

\numberwithin{equation}{section}

\definecolor{Gray}{gray}{0.5}

\def\cA{\mathcal{A}}

\def\cC{\mathcal{C}}

\def\cH{\mathcal{H}}

\def\cO{\mathcal{O}}
\def \cP{\mathcal{P}}
\def\cQ{\mathcal{Q}}

\def\AA{\mathbb{A}}

\def\CC{\mathbb{C}}

\def\FF{\mathbb{F}}

\def\QQ{\mathbb{Q}}
\def\RR{\mathbb{R}}

\def\TT{\mathbb{T}}

\def\ZZ{\mathbb{Z}}

\def\scA{\mathscr{A}}
\def\scB{\mathscr{B}}

\def\scH{\mathscr{H}}

\def\fh{\mathfrak{h}}

\def\fo{\mathfrak{o}}

\def\fq{\mathfrak{q}}

\def\dim{\mathrm{dim}}

\def\Hom{\mathrm{Hom}}

\def\tr{\mathrm{tr}}
\def\Tr{\mathrm{Tr}}

\def\wt{\mathrm{wt}}

\def\mapright#1{\smash{\mathop
        {\longrightarrow}\limits^{#1}}}

\def\la{\lambda}

\newcommand{\vect}[1]{\boldsymbol{#1}}

\def\norm{\|\hspace{-0.033cm}|}

\begin{document}
%
%
%
%
%
%
%
%
%
\title[Random walks on buildings]
 {A local limit theorem for random walks on the chambers of $\tilde{A}_2$ buildings}
\author{James Parkinson}

\address{
School of Mathematics and Statistics\\
University of Sydney, Australia}

\author{Bruno Schapira}
\address{D\'{e}partment de Math\'{e}matiques d'Orsay\\
Universit\'{e} Paris-Sud, France}



\date{\today}
\dedicatory{On the occasion of the birthdays of Professors Cartwright, Kaimanovich, and Picardello}

\begin{abstract}
In this paper we outline an approach for analysing random walks on the chambers of buildings. The types of walks that we consider are those which are well adapted to the structure of the building: Namely walks with transition probabilities $p(c,d)$ depending only on the Weyl distance $\delta(c,d)$. We carry through the computations for thick locally finite affine buildings of type $\tilde{A}_2$ to prove a local limit theorem for these buildings.
The technique centres around the representation theory of the associated Hecke algebra. This representation theory is particularly well developed for affine Hecke algebras, with elegant harmonic analysis developed by Opdam (\cite{opdamtrace}, \cite{opdamhanalysis}). We give an introductory account of this theory in the second half of this paper.
\end{abstract}

\maketitle

\section*{Introduction}

Probability theory on real Lie groups and symmetric spaces has a long and rich history (see \cite{bougerol}, \cite{guivarc'h}, \cite{guivarc'h2} and \cite{virtser} for example). A landmark work in the theory is Bougerol's 1981 paper \cite{bougerol} where the Plancherel Theorem of Harish-Chandra \cite{chandra} is applied to prove a local limit theorem for real semisimple Lie groups. There has also been considerable work done for Lie groups over local fields, such as $SL_n(\QQ_p)$ (see \cite{cartwrightwoess}, \cite{lindlbauer}, \cite{parkinsonwalks}, \cite{sawyer} and \cite{tolli} for example). In this case the group acts on a beautiful geometric object; the \textit{affine building}, and probability theory on the group can be analysed by studying probability theory on the building. It is this approach that we take here -- we develop a general setup for studying radial random walks on arbitrary buildings, and explicitly carry out the technique for $\tilde{A}_2$ buildings to prove a local limit theorem for random walks on the chambers of these buildings. 

A \textit{building} is a geometric/combinatorial object that can be defined axiomatically (see Definition~\ref{defn:buildings}). It is a set $\cC$ of \textit{chambers} (the rooms of the building) glued together in a highly structured way. The chambers can be visualised as simplices (all of the same dimension) and the gluing occurs along their codimension~$1$~faces, called \textit{panels}. Panels are the `doors' of the chambers -- one moves from chamber $c$ to chamber $d$ via the panel common to $c$ and $d$ (see Figure~\ref{fig:local}). Each panel $\pi$ has a \textit{type} $\mathrm{type}(\pi)$ (in some index set $I$) such that each chamber has exactly one panel of each type. If chambers $c$ and $d$ are glued together along their type $i$ panels then they are said to be $i$-\textit{adjacent}. There is a \textit{relative position function} $\delta(c,d)$ between any two chambers $c$ and $d$. This function takes values in a \textit{Coxeter group}~$W$ associated to the building, and it encodes the types of walks (or \textit{galleries}) in the building: If there is a minimal length walk from $c$ to $d$ passing through panels of types $i_1,\ldots,i_{\ell}$ then $\delta(c,d)=s_{i_1}\cdots s_{i_{\ell}}$ where $s_i$, $i\in I$, are the generators of the Coxeter group~$W$.

We will be considering random walks $A=(p(c,d))_{c,d\in\cC}$ on the set $\cC$ of chambers of a building. Here $p(c,d)$ is the probability that the walker moves from $c$ to $d$ in one step. A \textit{local limit theorem} is an asymptotic estimate for the \textit{$n$-step transition probability} $p^{(n)}(c,d)$ as $n\to\infty$, with $c$ and $d$ fixed.

Let us give a rough summary of the results and techniques of this paper. Let $(\cC,\delta)$ be a building with Coxeter group $W$. We will assume that $(\cC,\delta)$ satisfies a mild regularity condition (Definition~\ref{defn:regular}). Under this assumption the cardinality of the $w$-sphere $|\{d'\in\cC\mid\delta(c,d')=w\}|=q_w$ is independent of the centre $c\in\cC$ (for each $w\in W$). A random walk $A=(p(c,d))_{c,d\in\cC}$ is \textit{radial} if $p(c,d)=p(c',d')$ whenever $\delta(c,d)=\delta(c',d')$. It is elementary that a random walk $A=(p(c,d))_{c,d\in\cC}$ is radial if and only if
$$
A=\sum_{w\in W}a_wA_w\quad\textrm{where}\quad a_w\geq0\quad\textrm{and}\quad \sum_{w\in W}a_w=1,
$$
where $A_w=(p_w(c,d))_{c,d\in\cC}$ is the transition matrix for the random walk with transition probabilities
$$
p_w(c,d)=\begin{cases}\frac{1}{q_w}&\textrm{if $\delta(c,d)=w$}\\
0&\textrm{otherwise},
\end{cases}
$$
This naturally leads us to consider the linear span $\scA$ over $\CC$ of the operators $A_w$, $w\in W$. It is well known that $\scA$ is an algebra under convolution (see Proposition~\ref{prop:relations}). This algebra is the \textit{Hecke algebra} of the building; it is a noncommutative associative unital algebra. 

It is not difficult to see that if $A=(p(c,d))_{c,d\in\cC}$ is a radial random walk then
\begin{align}\label{eq:100}
p^{(n)}(c,d)=q_w^{-2}\Tr(A^nA_{w^{-1}})\qquad\textrm{if $\delta(c,d)=w$},
\end{align}
where $\Tr:\scA\to\CC$ is the \textit{canonical trace functional} given by linearly extending $\Tr(A_w)=\delta_{w,1}$. One can complete $\scA$ into a $C^*$-algebra~$\overline{\scA}$. Then $\Tr$ extends to a \textit{trace} on~$\overline{\scA}$. Under certain conditions on the representation theory of $\overline{\scA}$ (for example, liminality) there is general machinery on the decomposition of a trace that guarantees the existence of a unique Borel probability measure $\mu$ (the \textit{Plancherel measure}) such that (see \cite[\S8.8]{dixmier})
\begin{align}\label{eq:key}
\Tr(A)=\int_{\mathrm{spec}(\overline{\scA})}\chi_{\pi}(A)\,d\mu(\pi)\qquad\textrm{for all $A\in\overline{\scA}$}
\end{align}
where $\mathrm{spec}(\overline{\scA})$ is the \textit{spectrum} of $\overline{\scA}$, and $\chi_{\pi}$ is the character of the representation $\pi\in\mathrm{spec}(\overline{\scA})$ (we will be working in the situation where the irreducible representations are finite dimensional, and so $\chi_{\pi}(A)=\tr(\pi(A))$ where $\tr$ is the usual matrix trace). The usefulness of (\ref{eq:key}) for random walk theory is clear: If $A=(p(c,d))_{c,d\in\cC}$ is a radial random walk, then by (\ref{eq:100}) we have
\begin{align}\label{eq:start}
p^{(n)}(c,d)=q_w^{-2}\int_{\mathrm{spec}(\overline{\scA})}\chi_{\pi}(A^nA_{w^{-1}})\,d\mu(\pi)\qquad\textrm{if}\quad \delta(c,d)=w.
\end{align}
Therefore if we have a good understanding of $\mu$ and the representations $\pi$ in $\mathrm{spec}(\overline{\scA})$ then it should be possible to extract the leading behaviour of the integral as $n\to\infty$, thereby proving a local limit theorem. This is delicate work: The representation theory of Hecke algebras is a beautiful and rich subject with many subtleties. The representation theory is only really well developed in the cases where $W$ is finite or \textit{affine}. It is for this reason that in the end we will restrict ourselves to the affine case -- here we have at our disposal the elegant harmonic analysis of Opdam \cite{opdamhanalysis}. In fact we will restrict our specific computations to the $\tilde{A}_2$ case. The general affine case will appear elsewhere, where we also provide central limit theorems and rate of escape theorems. (See \cite{ramdiaconis} for an analysis of the finite case).

------

This paper is divided into Parts I and II, which can be more or less read independently. The local limit theorem appears in Part~I, and the derivation of the Plancherel formula is given in Part~II. Part~I also includes relevant background on Coxeter groups, buildings and the Hecke algebra of a building. Part~II contains relevant structure theory and representation theory of affine Hecke algebras, and an account of the harmonic analysis on affine Hecke algebras. The structural and representation theoretic results are well known, and the harmonic analysis results are from \cite{opdamtrace} and \cite{opdamhanalysis}, with some minor modifications. We make no claim of originality in Part~II, however we believe that this part is a nice contribution to the literature because it gives an introduction to the quite profound general analysis undertaken by Opdam (\cite{opdamtrace} and \cite{opdamhanalysis}).

------

Let us conclude this introduction by mentioning some related work. Brown and Diaconis \cite{diaconisbrown} and Billera, Brown and Diaconis \cite{billera} have studied random walks on hyperplane arrangements. This elegant theory is `just around the corner' from random walks on spherical (finite) buildings. Diaconis and Ram apply the representation theory of finite dimensional Hecke algebras to prove mixing time theorems for random walks on spherical buildings (see also \cite{diaconis}). In the context of affine buildings initial results came from the theory of homogeneous and semi-homogeneous trees (these are the $\tilde{A}_1$ buildings, arising from groups like $SL_2(\QQ_p)$). See \cite{sawyer}. The next simplest (irreducible) affine buildings are the $\tilde{A}_2$ buildings. Random walks on the vertices of these buildings are studied in \cite{lindlbauer} by Lindlbauer and Voit. Cartwright and Woess \cite{cartwrightwoess} study walks on the vertices of $\tilde{A}_d$ buildings and Parkinson \cite{parkinsonwalks} generalised this to walks on the vertices of arbitrary (regular) affine buildings. This work applies harmonic analysis from Macdonald \cite{macsph} and Matsumoto~\cite{matsumoto}. We note that the analysis on the vertices of an affine building is somewhat simpler than the chamber case, because the underlying Hecke algebra in the vertex case is commutative. Finally, Tolli \cite{tolli} has proved a local limit theorem for $SL_d(\QQ_p)$, which gives results for random walks on the associated building. 


\section*{Part I: The local limit theorem}

\section{Buildings and random walks}\label{sect:1}

Morally a \textit{building} is a way of organising the flag variety $G/B$ of a Lie group or Kac-Moody group into a geometric object that reflects the combinatorics of the \textit{Bruhat decomposition} and the internal structure of the double cosets~$BgB$. Remarkably buildings can be defined axiomatically, without any reference to the underlying connections with Lie groups and Kac-Moody groups. In this section we recall one of the axiomatic definitions of buildings. We define radial random walks on buildings, and write down the Hecke algebra of the building. Standard references for this section include \cite{bourbaki}, \cite{humphreys}, \cite{brown}, \cite{ronan} and \cite{woess}.

\subsection{Coxeter groups}

The notion of a \textit{Coxeter group} is at the heart of building theory.

\begin{defn}
A \textit{Coxeter system} is a pair $(W,S)$ where $W$ is a group generated by a finite set $S=\{s_0,\ldots,s_n\}$ subject to relations
$$
(s_is_j)^{m_{ij}}=1\qquad\textrm{for all $i,j=0,1,\ldots,n$},
$$
where (i) $m_{ii}=1$ for all $i$, (ii) $m_{ij}=m_{ji}$ for all $i,j$, and (iii) $m_{ij}\geq2$ is an integer or $\infty$ for $i\neq j$. We usually simply call $W$ a \textit{Coxeter group}.
\end{defn}

Coxeter groups are ``abstract reflection groups''. Indeed one can build a vector space on which $W$ acts by reflections (the \textit{reflection representation}). The relations $s_i^2=1$ say that $W$ is generated by reflections, and the relations $(s_is_j)^{m_{ij}}=1$ say that the product of the reflections $s_i$ and $s_j$ is a rotation by $2\pi/m_{ij}$.

\begin{defn} The \textit{length} $\ell(w)$ of $w\in W$ is
$$
\ell(w)=\min\{\ell\geq0\mid \textrm{$w$ can be written as a product of $\ell$ generators in $S$}\}.
$$
If $\ell(w)=\ell$ then an expression $w=s_{i_1}\cdots s_{i_{\ell}}$ is a \textit{reduced expression} for~$w$. It is easy to see that if $s_i\in S$ and $w\in W$ then $\ell(ws_i)=\ell(w)\pm 1$.
\end{defn}

\begin{example}\label{ex:A_2}
Consider the Coxeter system $(W,S)$ with $S=\{s_0,s_1,s_2\}$ and 
$$
s_0^2=s_1^2=s_2^2=(s_0s_1)^3=(s_1s_2)^3=(s_0s_2)^3=1.
$$
This is the Coxeter group of \textit{type} $\tilde{A}_2$, and it is the main example that we will consider in this work. This group can be realised nicely as a reflection group in $\mathbb{R}^2$.
 \begin{figure}[h]
 \begin{center}
       \includegraphics[totalheight=7.8cm]{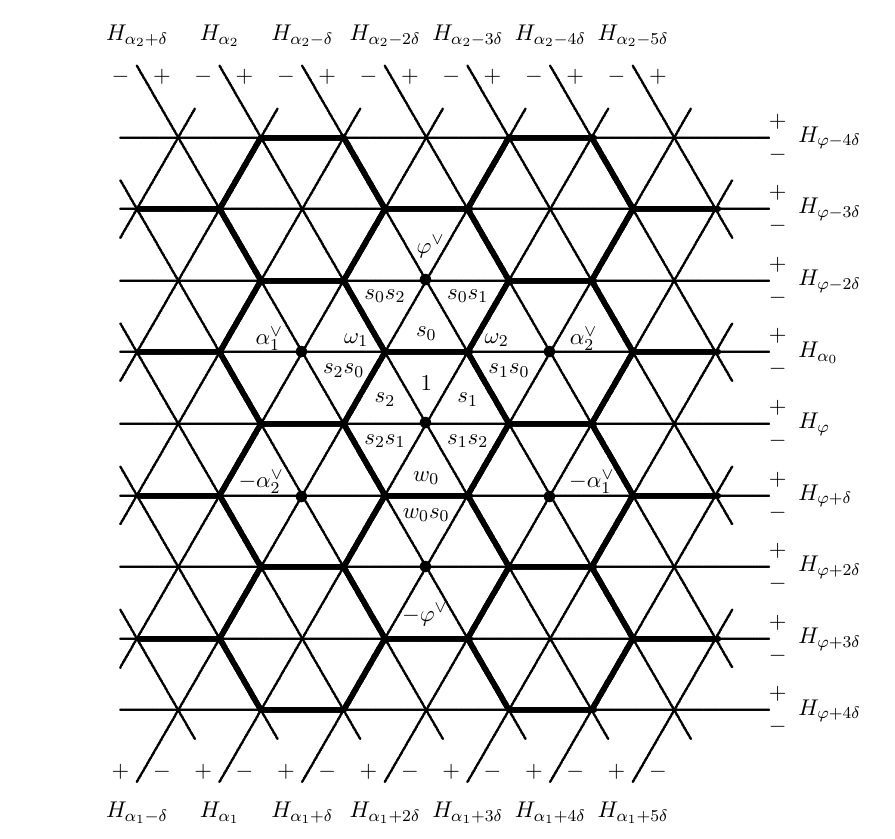}
 \end{center}
 \caption{The $\tilde{A}_2$ Coxeter group}\label{fig:A_2}
 \end{figure}
The elements $s_0,s_1,s_2$ are the reflections in the hyperplanes labeled by $H_{\alpha_0}$, $H_{\alpha_1}$ and $H_{\alpha_2}$. Then $W$ acts simply transitively on the set of triangles. In the building language these triangles are called \textit{chambers}, and in some other aspects of Lie theory they are called \textit{alcoves}. The remainder of the details are explained later. 
\end{example}

\subsection{Buildings}

We adopt the following modern definition of a building, from \cite{brown}. 

\begin{defn}\label{defn:buildings}
A \textit{building of type $(W,S)$} is a pair $(\cC,\delta)$ consisting of a nonempty set $\cC$ of \textit{chambers}, together with a map $\delta:\cC\times\cC\to W$ such that for all $a,b,c\in\cC$:
\begin{enumerate}
\item[\textbf{(B1)}] $\delta(a,b)=1$ if and only if $a=b$.
\item[\textbf{(B2)}] If $\delta(a,b)=w$ and $\delta(b,c)=s_i$, then $\delta(a,c)\in\{w,ws_i\}$. If $\ell(ws_i)=\ell(w)+1$ then $\delta(a,c)=ws_i$.
\item[\textbf{(B3)}] If $\delta(a,b)=w$ then for each $s_i$ there is a chamber $c'\in\cC$ with $\delta(b,c')=s_i$ such that $\delta(a,c')=ws_i$. This chamber is unique if $\ell(ws_i)=\ell(w)-1$.
\end{enumerate}
The function $\delta:\cC\times\cC\to W$ is the \textit{Weyl distance function}. It follows from the axioms that $\delta(a,b)=\delta(b,a)^{-1}$.
\end{defn}

One can visualise the building geometrically as follows. For simplicity, let us suppose that $S=\{s_0,s_1,s_2\}$ (like in the $\tilde{A}_2$ example). Then each chamber of the building is imagined as a triangle, with the sides (codimension~1 faces) being called \textit{panels}. Each panel $\pi$ is assigned a \textit{type} $\textrm{type}(\pi)\in\{0,1,2\}$ such that every chamber has exactly one panel of each type. If chambers $c,d\in\cC$ have $\delta(c,d)=s_i$ then we glue the chambers $c$ and $d$ together along the type $i$ panels. Therefore the local picture of the building looks like Figure~\ref{fig:local}.
\begin{figure}[h]
 \begin{center}
        \includegraphics[totalheight=3.5cm]{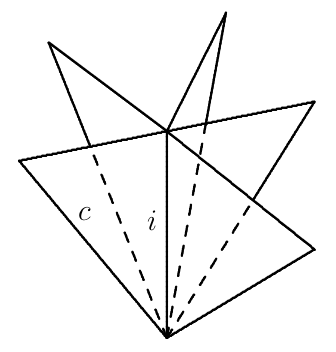}
 \end{center}\vspace{-0.4cm}
 \caption{The local view of a rank~$3$ building}\label{fig:local}
\end{figure}

One calls chambers $c$ and $d$ \textit{$i$-adjacent} if $\delta(c,d)=s_i$ or $c=d$. This is an equivalence relation, and we write $c\sim_i d$ if $c$ and $d$ are $i$-adjacent. Figure~\ref{fig:local} shows the set of all chambers $i$-adjacent to~$c$. A \textit{gallery of type $i_1\cdots i_{\ell}$ from~$c$ to~$d$} is a sequence $(c_0,c_1,\ldots,c_{\ell})$ of chambers with
$$
c=c_0\sim_{i_1} c_1\sim_{i_2}\cdots\sim_{i_{\ell}}c_{\ell}=d,\qquad\textrm{with \quad$c_{k-1}\neq c_k$ \quad for $k=1,\ldots,\ell$}.
$$
So a gallery is a ``walk'' from chamber to chamber through the building. One can show that if $w=s_{i_1}\cdots s_{i_{\ell}}$ is a reduced expression then
$$
\delta(c,d)=w\quad\Longleftrightarrow\quad \textrm{there is a minimal length gallery of type $i_1\cdots i_{\ell}$ from $c$ to $d$}.
$$
So the Weyl distance encodes the types of the minimal length galleries from $c$ to~$d$.

\begin{defn}
Let $w\in W$ and $c\in\cC$. The \textit{$w$-sphere centered at $c$} is
$$
\cC_w(c)=\{d\in\cC\mid\delta(c,d)=w\}.
$$
In particular if $s_i\in S$ then $\cC_{s_i}(c)=\{d\in\cC\mid c\sim_i d\}\backslash\{c\}$.
\end{defn}

Therefore if $w=s_{i_1}\cdots s_{i_{\ell}}$ is a reduced expression then $\cC_w(c)$ is the set of all chambers in the building that are connected to $c$ by a gallery of type $i_1\cdots i_{\ell}$. 

\begin{defn}\label{defn:regular}
A building $(\cC,\delta)$ with Coxeter system $(W,S)$ is: 
\begin{enumerate}
\item[$\bullet$] \textit{thin} if $|\cC_s(c)|=1$ for all $s\in S$ and $c\in\cC$,
\item[$\bullet$] \textit{thick} if $|\cC_s(c)|\geq2$ for all $s\in S$ and $c\in\cC$, 
\item[$\bullet$] \textit{locally finite} if $|\cC_{s}(c)|<\infty$ for all $s\in S$ and $c\in\cC$,
\item[$\bullet$] \textit{regular} if for each $s\in S$,
$|\cC_{s}(c)|=|\cC_{s}(d)|$ for all chambers $c,d\in\cC$.
\end{enumerate}
If $(\cC,\delta)$ is a locally finite regular building then we define $q_0,\ldots,q_n\in\ZZ_{>0}$ by $q_i=|\cC_{s_i}(c)|$ for any $c\in\cC$. The integers $q_0,\ldots,q_n$ are called the \textit{parameters} of the building. For example if Figure~\ref{fig:local} represents part of a locally finite regular building then $q_i=4$ (there are $5=4+1$ chambers on each $i$-panel).
\end{defn}
\begin{center}
\textit{Henceforth we will assume that our buildings are locally finite and regular.}
 \end{center}
  
 \begin{remark} If $(\cC,\delta)$ is thick and locally finite and if $m_{ij}<\infty$ for each $i,j$ then by \cite[Theorem~2.4]{P1} $(\cC,\delta)$ is regular. So regularity is a very weak hypothesis.
\end{remark}
 
 A simple induction shows that if $w=s_{i_1}\cdots s_{i_{\ell}}$ is a reduced expression then
\begin{align}\label{eq:sphere}
|\cC_w(c)|=q_{i_1}\cdots q_{i_{\ell}}\qquad\textrm{for all $c\in\cC$}.
\end{align}
Thus we can define $q_w=q_{i_1}\cdots q_{i_{\ell}}$ (equation (\ref{eq:sphere}) shows that this is independent of the particular reduced expression for~$w$). Since $s_is_js_i\cdots=s_js_is_j\cdots$ ($m_{ij}$ factors on each side) are both reduced expressions it follows that $q_i=q_j$ whenever $m_{ij}$ is finite and odd. Then it follows that if $s_j=ws_iw^{-1}$ for some $w\in W$ then $q_i=q_j$ (see \cite[IV, \S1, No.3, Proposition~3]{bourbaki}).

\begin{remark} Given a Coxeter system $(W,S)$, define a building $(W,\delta_W)$ where $\delta_W(u,v)=u^{-1}v$. Figure~\ref{fig:A_2} shows the $\tilde{A}_2$ case. The building $(W,\delta_W)$ is thin, and all thin buildings arise in this way. A general building of type $(W,S)$ contains many thin sub-buildings of type $(W,S)$. These sub-buildings are called the \textit{apartments} of the building. The apartments fit together in a highly structured way:
\begin{enumerate}
\item[\textbf{(A1)}] Given chambers $c,d\in\cC$ there exists an apartment containing both.
\item[\textbf{(A2)}] If $A$ and $A'$ are apartments with $A\cap A'\neq \emptyset$ then there is an isomorphism $\psi:A\to A'$ fixing each chamber of the intersection $A\cap A'$.
\end{enumerate}
These facts give a \textit{global picture} of a building (see Figure~\ref{fig:A_2building}).
\begin{figure}[h]
\begin{center}
        \includegraphics[totalheight=4cm]{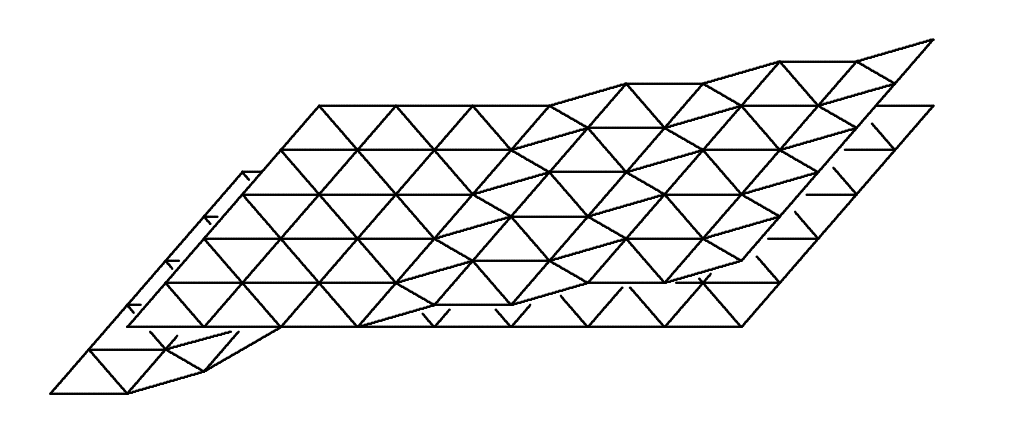}
 \end{center}
 \caption{The global view of an $\tilde{A}_2$ building}\label{fig:A_2building}
 \end{figure} 

\noindent Note that the apartments are as in Figure~\ref{fig:A_2}; there are $6$ apartments shown in Figure~\ref{fig:A_2building}. However if the building is thick then the ``branching'' is actually happening along all of the walls of the building. Therefore a thick $\tilde{A}_2$ building has infinitely many apartments. To understand buildings it is useful to have both the local and global pictures in mind.
\end{remark}

\begin{remark} A locally finite regular $\tilde{A}_2$ building necessarily has $q_0=q_1=q_2=q$ because $m_{0,1}=m_{1,2}=m_{2,0}=3$ are odd. An $\tilde{A}_2$ building is \textit{not} determined by its thickness parameter $q$. For example the buildings constructed from $SL_3(\QQ_p)$ and $SL_3(\FF_p((t)))$ are non-isomorphic and both have thickness $q=p$. Furthermore it is unknown which parameters $q$ can occur as the thickness of an $\tilde{A}_2$ building. By \cite{ronanbuild} this is closely related to the famous unsolved problem of classifying finite projective planes. 
\end{remark}

\begin{remark}
The definition of buildings is driven by the combinatorics of \textit{Kac-Moody groups}, which are infinite dimensional generalisations of semisimple Lie groups. If $G$ is a Kac-Moody group with Borel subgroup $B$ and Weyl group $W$ then the flag variety $G/B$ is a building with $\delta(gB,hB)=w$ if and only if $g^{-1}hB\subseteq BwB$.
\end{remark}

\subsection{Random walks and the Hecke algebra}

A \textit{random walk} consists of a finite or countable \textit{statespace} $X$ and a \textit{transition operator} $A=(p(x,y))_{x,y\in X}$ where $p(x,y)\geq0$ for all $x$ and $y$ and $\sum_{y\in X}p(x,y)=1$ for all $x\in X$. As an operator acting on the space of functions $f:X\to\CC$ we have
$$
(Af)(x)=\sum_{y\in X}p(x,y)f(y)\qquad\textrm{for all $f:X\to\CC$ and $x\in X$}.
$$ 
The numbers $p(x,y)$ are the \textit{transition probabilities} of the walk. The natural interpretation of a random walk is that of a walker taking discrete steps in the space~$X$, with $p(x,y)$ being the probability that the walker, having started at $x$, moves to $y$ in one step. The \textit{$n$-step transition probability} $p^{(n)}(x,y)$ is the probability that the walker, having started at $x$, is at $y$ after $n$ steps. Then $A^n=(p^{(n)}(x,y))_{x,y\in X}$. A \textit{local limit theorem} is a theorem giving an asymptotic estimate for $p^{(n)}(x,y)$ as $n\to\infty$ (with $x,y\in X$ fixed).

Here we consider random walks with statespace $\cC$ (the set of chambers of a building). We consider random walks which are well adapted to the structure of the building:

\begin{defn} A random walk $A=(p(c,d))_{c,d\in\cC}$ on the chambers of a building $(\cC,\delta)$ is \textit{radial} if 
$p(c,d)=p(c',d')$ whenever $\delta(c,d)=\delta(c',d')$.
\end{defn}

Recall that we assume our buildings are locally finite and regular, and so $|\cC_w(c)|=q_w$. For each $w\in W$, the \textit{$w$-averaging operator} is
$$
(A_wf)(c)=\frac{1}{q_{w}}\sum_{d\in\cC_w(c)}f(d)\qquad\textrm{for $f:\cC\to \CC$ and $c\in\cC$}.
$$
Then $A_w=(p_w(c,d))_{c,d\in\cC}$ is the transition operator of the radial walk with 
$$
p_w(c,d)=\begin{cases}q_w^{-1}&\textrm{if $\delta(c,d)=w$}\\
0&\textrm{otherwise.}
\end{cases}$$

The following proposition is elementary.

\begin{prop}\label{prop:radial} A random walk $A=(p(c,d))_{c,d\in\cC}$ is radial if and only if
$$
A=\sum_{w\in W}a_wA_w\qquad\textrm{where\quad $a_w\geq0$\quad and }\quad\sum_{w\in W}a_w=1,
$$
in which case $p(c,d)=a_wq_w^{-1}$ if $\delta(c,d)=w$.
\end{prop}

Therefore we are naturally lead to consider linear combinations of the (linearly independent) operators $A_w, w\in W$. Let $\scA$ be the vector space over $\CC$ with basis $\{A_w\mid w\in W\}$. The following simple proposition tells us how to compose the averaging operators, and shows that $\scA$ is an algebra.

\begin{prop}\label{prop:relations} Let $w\in W$ and $s_i\in S$. The averaging operators satisfy
$$
A_wA_{s_i}=\begin{cases}A_{ws_i}&\textrm{if $\ell(ws_i)=\ell(w)+1$}\\
q_i^{-1}A_{ws_i}+\left(1-q_i^{-1}\right)A_w&\textrm{if $\ell(ws_i)=\ell(w)-1$}.
\end{cases}
$$
Therefore $\scA$ is an algebra. 
\end{prop}

\begin{proof}
Using the definition of the operators we see that
\begin{align*}
(A_wA_{s_i}f)(c)&=\frac{1}{q_wq_i}\sum_{d\in \cC_w(c)}\sum_{e\in\cC_{s_i}(d)}f(e)=\frac{1}{q_wq_i}\sum_{e\in\cC}|\cC_w(c)\cap\cC_{s_i}(e)|f(e).
\end{align*}
If $\cC_w(c)\cap\cC_{s_i}(e)\neq\emptyset$ then (\textbf{B2}) implies that $e\in\cC_w(c)$ or $e\in\cC_{ws_i}(c)$. Then:
\begin{align*}
&\textrm{If $e\in\cC_w(c)$,}&&\quad|\cC_w(c)\cap \cC_{s_i}(e)|=\begin{cases}0&\textrm{if $\ell(ws_i)=\ell(w)+1$ (by (\textbf{B2}))}\\
q_i-1&\textrm{if $\ell(ws_i)=\ell(w)-1$ (by (\textbf{B3}))}
\end{cases}\\
&\textrm{If $e\in\cC_{ws_i}(c)$,}&&\quad|\cC_w(c)\cap \cC_{s_i}(e)|=\begin{cases}
1&\textrm{if $\ell(ws_i)=\ell(w)+1$ (by (\textbf{B3}))}\\
q_i&\textrm{if $\ell(ws_i)=\ell(w)-1$ (by (\textbf{B2}))}.
\end{cases}
\end{align*}
Therefore if $\ell(ws_i)=\ell(w)-1$ we have
\begin{align*}
(A_wA_{s_i}f)(c)=\frac{q_{ws_i}}{q_w}(A_{ws_i}f)(c)+(1-q_i^{-1})(A_wf)(c).
\end{align*}
Since $\ell(ws_i)=\ell(w)-1$ we have $q_w=q_{(ws_i)s_i}=q_{ws_i}q_i$. This completes the proof when $\ell(ws_i)=\ell(w)-1$, and the case $\ell(ws_i)=\ell(w)+1$ is similar. Now a simple induction on $\ell(v)$ shows that $A_uA_v$ is a linear combination of terms $A_w$, $w\in W$. Therefore $\scA$ is an algebra.
\end{proof}

\begin{defn} The algebra $\scA$ is the \textit{Hecke algebra} of the building $(\cC,\delta)$. \end{defn}

If a radial walk $A=(p(c,d))_{c,d\in\cC}$ is written as $A=\sum a_wA_w$ as in Proposition~\ref{prop:radial} then the $n$-step transition probabilities $p^{(n)}(c,d)$ can be found from the following calculation:
\begin{align}\label{eq:radial}
p^{(n)}(c,d)=a_w^{(n)}q_w^{-1},\quad\textrm{where}\quad
A^n=\bigg(\sum_{w\in W}a_wA_w\bigg)^n=\sum_{w\in W}a_w^{(n)}A_w.
\end{align}
So finding $p^{(n)}(c,d)$ is the equivalent to finding the coefficient $a_w^{(n)}$ of $A_w$ in $A^n$.

\section{The Plancherel Theorem}\label{sect:2}

 In this section we discuss how the representation theory of the Hecke algebra can be used to achieve the goal of computing $p^{(n)}(c,d)$. The representation theory of Hecke algebras is particularly well developed in two important cases: When the underlying Coxeter group is a \textit{finite Weyl group} or an \textit{affine Weyl group}. In the case of a finite Weyl group the building is a finite object, and the types of questions one asks are quite different to what we do here (see~\cite{ramdiaconis}). Therefore we will focus on the affine case here (but our initial setup will remain rather general).

\subsection{The Hecke algebra as a $C^*$-algebra}\label{sect:hanalysisA}

 Let $\ell^2(\cC)$ be the space of square summable functions $f:\cC\to\CC$, with inner product $\langle f,g\rangle=\sum f(c)\overline{g(c)}$. Each $A\in\scA$ maps $\ell^2(\cC)$ into itself (c.f. \cite[Lemma 4.1]{cartsph}):

\begin{lemma} Let $w\in W$. If $f\in\ell^2(\cC)$ then $A_wf\in\ell^2(\cC)$ and $\|A_w\|\leq 1$, where $
\|A\|=\sup\{\|Af\|_2\,:\,f\in\ell^2(\cC),\|f\|_2\leq1\}
$ is the $\ell^2$-operator norm of $A\in\scA$.
\end{lemma}

Therefore $\scA$ is a subalgebra of the $C^*$-algebra $\scB(\ell^2(\cC))$ of bounded linear operators on $\ell^2(\cC)$, and since $A_w^*=A_{w^{-1}}$ we see that $\scA$ is closed under the adjoint involution.  Let $\overline{\scA}$ denote the completion of $\scA$ with respect to the $\ell^2$-operator norm. Therefore $\overline{\scA}$ is a (non-commutative) $C^*$-algebra.

Let $o\in\cC$ be a fixed chamber. Since $A_w\delta_o=q_w^{-1}1_{\cC_{w^{-1}}(o)}$ it follows that $\langle A_u\delta_o,A_v\delta_o\rangle=\delta_{u,v}q_u^{-1}$. Let
$$
(A,B):=\langle A\delta_o,B\delta_o\rangle\qquad\textrm{for $A,B\in\overline{\scA}$}.
$$
The value of $(A,B)$ does not depend on the particular fixed chamber $o\in\cC$. Moreover, $(\cdot,\cdot)$ defines an inner product on $\overline{\scA}$. The only thing to check is:

\begin{lemma}\label{lem:small} Let $A\in\overline{\scA}$. If $A\delta_o=0$ then $A=0$.
\end{lemma}

\begin{proof}
This is easily checked for $A\in\scA$, and thus is true for $A\in\overline{\scA}$ by density.
\end{proof}

It is routine to verify the following properties:
\begin{align}\label{eq:innerproduct}
(AB,C)=(B,A^*C)\quad\textrm{and}\quad (A,B)=(B^*,A^*)\quad \textrm{for all $A,B,C\in\overline{\scA}$}.
\end{align}

\subsection{The trace functional}

Let $o\in\cC$ be a fixed chamber. The linear map
$$
\Tr:\overline{\scA}\to\CC\quad\textrm{with}\quad \Tr(A)=(A\delta_o)(o)=(A,I)
$$
defines a \textit{trace} on $\overline{\scA}$, because by (\ref{eq:innerproduct}) we have
$$
\Tr(AB)=(AB,I)=(B,A^*)=(A,B^*)=(BA,I)=\Tr(BA).
$$
Note that $\Tr(\sum a_wA_w)=a_1$ and $\Tr(A_u^*A_v)=(A_v,A_u)=q_u^{-1}\delta_{u,v}$, and so (\ref{eq:100}) follows from~(\ref{eq:radial}).

There is a general theory centred around decomposing a trace on a liminal $C^*$-algebra into an integral over irreducible $*$-representations of the algebra (the \textit{spectral decomposition}). For an elegant account see \cite[\S8.8]{dixmier}. A $C^*$-algebra $\cA$ is \textit{liminal} if for every irreducible representation $\pi$ of $\cA$ and for each $x\in \cA$ the operator $\pi(x)$ is compact. Suppose that $\overline{\scA}$ is liminal; indeed this is true if $(\cC,\delta)$ is affine because all of the irreducible representations are finite dimensional (see Proposition~\ref{prop:fundamental}). Then by \cite[\S8.8]{dixmier} there exists a unique Borel probability measure $\mu$ (the \textit{Plancherel measure}) such that (\ref{eq:key}) holds. The way we plan to apply (\ref{eq:key}) was explained in the introduction.

\subsection{Statement of the Plancherel Theorem for type $\tilde{A}_2$}\label{sect:statement}

By the \textit{Plancherel Theorem} we mean the computation of the measure $\mu$ and the spectrum $\mathrm{spec}(\overline{\scA})$ in (\ref{eq:key}). Let us state the Plancherel Theorem for Hecke algebras of type $\tilde{A}_2$. Since this is a representation theoretic statement it is first essential to write down some representations of $\scA$. See Section~\ref{sect:8} for the details. Recall that in type $\tilde{A}_2$ we have $q_0=q_1=q_2=q$.

\smallskip

\noindent\textbf{A  6-dimensional representation:} For each $t=(t_1,t_2)$ with $t_1,t_2\in\CC^{\times}$ there is a 6-dimensional representation $\pi_t:\scA\to M_6(\CC)$ given on the generators of~$\scA$ by the matrices
$$
\pi_t(A_0)=\frac{1}{\sqrt{q}}\begin{pmatrix}
\fq&0&0&0&0&t_1t_2\\
0&\fq&0&t_2&0&0\\
0&0&\fq&0&t_1&0\\
0&t_2^{-1}&0&0&0&0\\
0&0&t_1^{-1}&0&0&0\\
t_1^{-1}t_2^{-1}&0&0&0&0&0\end{pmatrix},
$$
$$
\pi_t(A_1)=\frac{1}{\sqrt{q}}\begin{pmatrix}
0&1&0&0&0&0\\
1&\fq&0&0&0&0\\
0&0&0&1&0&0\\
0&0&1&\fq&0&0\\
0&0&0&0&0&1\\
0&0&0&0&1&\fq
\end{pmatrix},\quad\pi_t(A_2)=\frac{1}{\sqrt{q}}\begin{pmatrix}
0&0&1&0&0&0\\
0&0&0&0&1&0\\
1&0&\fq&0&0&0\\
0&0&0&0&0&1\\
0&1&0&0&\fq&0\\
0&0&0&1&0&\fq
\end{pmatrix},
$$
where for type-setting convenience $\fq=q^{\frac{1}{2}}-q^{-\frac{1}{2}}$. This representation is the \textit{principal series representation} of $\scA$ with \textit{central character}~$t=(t_1,t_2)$. It is irreducible if and only if $t_1,t_2\neq q^{\pm1}$, and every irreducible representation of $\scA$ is a composition factor of a principal series representation for some central character~$t$.

\smallskip

\noindent\textbf{A 3-dimensional representation:} For each $u\in\CC^{\times}$ there is a 3-dimensional representation $\pi_u^{(1)}:\scA\to M_3(\CC)$ given on the generators of $\scA$ by the matrices
$$
\pi_u^{(1)}(A_0)=\frac{1}{\sqrt{q}}\begin{pmatrix}
\fq&0&-u\\
0&-q^{-\frac{1}{2}}&0\\
-u^{-1}&0&0\end{pmatrix},
$$
$$
\pi_u^{(1)}(A_1)=\frac{1}{\sqrt{q}}\begin{pmatrix}
-q^{-\frac{1}{2}}&0&0\\
0&0&1\\
0&1&\fq\end{pmatrix},\qquad
\pi_u^{(1)}(A_2)=\frac{1}{\sqrt{q}}\begin{pmatrix}
0&1&0\\
1&\fq&0\\
0&0&-q^{-\frac{1}{2}}
\end{pmatrix}
$$
This representation is an \textit{induced representation}, constructed by lifting a representation of a \textit{parabolic subalgebra} of $\scA$ to the full algebra. 

\smallskip

\noindent\textbf{A $1$-dimensional representation:} There is a 1-dimensional representation of $\scA$ $\pi^{(2)}:\scA\to\CC$ given on the generators of $\scA$ by 
$$
\pi^{(2)}(A_0)=\pi^{(2)}(A_1)=\pi^{(2)}(A_2)=-q^{-1}.
$$

\smallskip

It can be shown that all of the above representations extend to $\overline{\scA}$. The details will be provided elsewhere in a more general setting. We can now state the Plancherel Theorem for type $\tilde{A}_2$. Let $\TT$ be the circle group 
$$
\TT=\{t\in\CC\mid |t|=1\},
$$
and let $dt$ be normalised Haar measure on $\TT$. Let $\chi_t$, $\chi_u^{(1)}$ and $\chi^{(2)}$ be the characters of $\pi_t,\pi_u^{(1)}$ and $\pi^{(2)}$ respectively. For example, $\chi_t(A)=\tr(\pi_t(A))$, where $\tr$ is the usual matrix trace on $M_6(\CC)$.

\begin{thm}\label{thm:plancherel} Let $(\cC,\delta)$ be a thick locally finite regular $\tilde{A}_2$ building. Then
$$
\Tr(A)=\frac{1}{6q^3}\int_{\TT^2}\frac{\chi_{t}(A)}{|c(t)|^2}\,dt_1dt_2+\frac{(q-1)^2}{q^2(q^2-1)}\int_{\TT}\frac{\chi_{u}^{(1)}(A)}{|c_1(u)|^2}\,du+\frac{(q-1)^3}{q^3-1}\chi^{(2)}(A)
$$
for all $A\in\overline{\scA}$, where
\begin{align*}
c(t)=\frac{(1-q^{-1}t_1^{-1})(1-q^{-1}t_2^{-1})(1-q^{-1}t_1^{-1}t_2^{-1})}{(1-t_1^{-1})(1-t_2^{-1})(1-t_1^{-1}t_2^{-1})}\quad \textrm{and}\quad c_1(u)=\frac{1-q^{-\frac{3}{2}}u^{-1}}{1-q^{\frac{1}{2}}u^{-1}}.
\end{align*}
\end{thm}

\begin{proof} See Section~\ref{sect:7}.
\end{proof}

\section{The local limit theorem}\label{sect:3}

Let $(\cC,\delta)$ be a locally finite thick $\tilde{A}_2$ building. Therefore $(\cC,\delta)$ is necessarily regular, and $q_0=q_1=q_2=q\geq2$. Let $P=\frac{1}{3}(A_0+A_1+A_2)$ be the transition operator for the simple random walk on $(\cC,\delta)$. That is $P=(p(c,d))_{c,d\in\cC}$ with
$$
p(c,d)=\begin{cases}
\frac{1}{3q}&\textrm{if $c\sim d$ and $c\neq d$}\\
0&\textrm{otherwise}.
\end{cases}
$$
This walk is irreducible (because $\{s_0,s_1,s_2\}$ generates $W$) and aperiodic (this follows from $A_i^2=q^{-1}+(1-q^{-1})A_i$). Our techniques will work for general radial random walks (not just the simple random walk), but the additional generality requires a more careful study of the representation theory of affine Hecke algebras to obtain the bounds and estimates required to make the analysis work. We have chosen to deal with this in a later work, where walks on general affine buildings are studied.

If $\theta=(\theta_1,\theta_2)\in\RR^2$ we write $e^{i\theta}=(e^{i\theta_1},e^{i\theta_2})\in\TT^2$. For $\theta\in\RR^2$ and $\varphi\in\RR$ the matrices $\pi_{e^{i\theta}}(P)$, $\pi_{e^{i\varphi}}^{(1)}(P)$ and $\pi^{(2)}(P)$ are given by 
$$
\pi_{e^{i\theta}}(P)=\frac{1}{3\sqrt{q}}\begin{pmatrix}
\fq&1&1&0&0&e^{i(\theta_1+\theta_2)}\\
1&2\fq&0&e^{i\theta_2}&1&0\\
1&0&2\fq&1&e^{i\theta_1}&0\\
0&e^{-i\theta_2}&1&\fq&0&1\\
0&1&e^{-i\theta_1}&0&\fq&1\\
e^{-i(\theta_1+\theta_2)}&0&0&1&1&2\fq
\end{pmatrix},
$$
$$
\pi_{e^{i\varphi}}^{(1)}(P)=\frac{1}{3\sqrt{q}}\begin{pmatrix}q^{\frac{1}{2}}-2q^{-\frac{1}{2}}&1&-e^{i\varphi}\\
1&q^{\frac{1}{2}}-2q^{-\frac{1}{2}}&1\\
-e^{-i\varphi}&1&q^{\frac{1}{2}}-2q^{-\frac{1}{2}}\end{pmatrix}\quad\textrm{and}\quad \pi^{(2)}(P)=-q^{-\frac{3}{2}}
$$
where as before $\fq=q^{\frac{1}{2}}-q^{-\frac{1}{2}}$.

The local limit theorem requires a careful study of the eigenvalues of $\pi_{e^{i\theta}}(P)$ for $(\theta_1,\theta_2)$ close to~$(0,0)$. Let $\la_1(\theta)\geq\cdots\geq\la_6(\theta)$ be the eigenvalues of $\pi_{e^{i\theta}}(P)$. Let $\la_i=\la_i(0)$. Let $\mu_1(\varphi)\geq\mu_2(\varphi)\geq\mu_3(\varphi)$ be the eigenvalues of $\pi_{e^{i\varphi}}^{(1)}(P)$, where $\varphi\in\RR$. All of these eigenvalues are real, because the matrices are Hermitian.

Explicit formulae for the eigenvalues are not feasible, and so we turn to techniques from perturbation theory. Standard references include~\cite{baumgaertel} and~\cite{katot}. For perturbation theory to work nicely one wants to have complete eigenvalue and eigenvector information for $\pi_1(P)$. The eigenvalues of $\pi_1(P)$ are easily computed. In decreasing order of magnitude they are $1>\la_1>\la_2=\la_3>\la_4=\la_5>\la_6>0$ with $\la_1,\la_2,\la_4$ and $\la_6$ given by
$$
\frac{3(q-1)+\sqrt{q^2+34q+1}}{6q},\,\frac{2(q-1)}{3q},\,\frac{q-1}{3q},\,\frac{3(q-1)-\sqrt{q^2+34q+1}}{6q}
$$
respectively. The eigenspaces $\vect{e}(\la)$ are
\begin{align*}
\begin{aligned}
\vect{e}(\la_1)&=\CC(a,1,1,a,a,1),\\
\vect{e}(\la_6)&=\CC(-b,1,1,-b,-b,1),
\end{aligned}\quad
\begin{aligned}
\vect{e}(\la_2)&=\CC(-1,0,0,1,0,0)+\CC(-1,0,0,0,1,0),\\
\vect{e}(\la_4)&=\CC(0,-1,1,0,0,0)+\CC(0,-1,0,0,0,1)
\end{aligned}
\end{align*}
where $a=\frac{\sqrt{q^2+34q+1}-(q-1)}{6\sqrt{q}}$ and $b=\frac{q-1+\sqrt{q^2+34q+1}}{6\sqrt{q}}$. Let $\vect{v}_1(\theta),\ldots,\vect{v}_6(\theta)$ be an orthonormal basis of $\CC^6$ with $\vect{v}_i(\theta)$ a $\la_i(\theta)$-eigenvector.

\begin{remark} Perron-Frobenius guarantees that the largest eigenvalue of $\pi_1(P)$ is simple with a positive eigenvector (note that $\pi_1(P)^2$ has all entries positive).
\end{remark}

The eigenvalues of $\pi_1^{(1)}(P)$ are
$
\frac{1}{3\sqrt{q}}(q^{\frac{1}{2}}+2q^{-\frac{1}{2}}-1)$ and $\frac{1}{3\sqrt{q}}(q^{\frac{1}{2}}-2q^{-\frac{1}{2}}-2)
$
with the first eigenvalue repeated.

\begin{lemma}\label{lem:bound1}
Let $A=\sum a_wA_w\in\scA$ with $a_w\geq0$. Then
\begin{align*}
|\chi_{e^{i\theta}}(A)|\leq \chi_1(A)\qquad\textrm{for all $\theta\in\RR^2$}.
\end{align*}
\end{lemma}

\begin{proof} The proof uses some of the general representation theory from Section~\ref{sect:5}. It follows from Theorem~\ref{thm:combform} that $\chi_{e^{i\theta}}(A_w)$ is a linear combination of terms $\{e^{ik\theta_1}e^{i\ell\theta_2}\mid k,\ell\in\ZZ\}$ with nonnegative coefficients. Therefore $\chi_{e^{i\theta}}(A)$ also has this property, and the result follows.
\end{proof}

In the proof of the following lemma we will use some well known inequalities between the eigenvalues of the sum of Hermitian matrices (see the interesting survey~\cite{fulton}). In particular, if $X$ and $Y$ are arbitrary $d\times d$ Hermitian matrices with eigenvalues $x_1\geq\cdots\geq x_d$ and $y_1\geq\cdots\geq y_d$ and if $z_1\geq\cdots\geq z_d$ are the eigenvalues of $Z=X+Y$ then 
$$
z_1+\cdots+z_r\leq x_1+\cdots+x_r+y_1+\cdots+y_r\quad\textrm{for each $1\leq r\leq d$}.
$$
It follows that
\begin{align}\label{eq:eigenvalueinequalities}
z_r\leq x_1+y_1\qquad\textrm{and}\qquad z_r\geq x_d+y_d\qquad\textrm{for all $1\leq r\leq d$}
\end{align}
(for the second inequality use the trace identity $\tr(Z)=\tr(X)+\tr(Y)$).

\begin{lemma}\label{lem:boundfull} We have the following.
\begin{enumerate}
\item $|\la_i(\theta)|\leq \la_1$ with equality if and only if $i=1$ and $\theta_1,\theta_2\in 2\pi\ZZ$.
\item $|\mu_i(\varphi)|<\la_1$ for all $i=1,2,3$ and all $\varphi\in\RR$.
\item $|\chi^{(2)}(P)|<\la_1$.
\end{enumerate}
\end{lemma}

\begin{proof}
1. If $|\la_i(\theta)|>\la_1$ then $|\chi_{e^{i\theta}}(P^k)|>\chi_1(P^k)$ for sufficiently large $k$, contradicting Lemma~\ref{lem:bound1}. Therefore $|\la_i(\theta)|\leq\la_1$ for all $i=1,\ldots,6$ and all $\theta\in\RR^2$. Suppose that $|\la_i(\theta)|=\la_1$. Writing $\pi_{e^{i\theta}}(P)=\pi_1(P)+E(\theta)$ we see that $E(\theta)$ has eigenvalues $\pm\frac{2}{3\sqrt{q}}|\sin\frac{\theta_1}{2}|,\pm\frac{2}{3\sqrt{q}}|\sin\frac{\theta_2}{2}|$ and $\pm\frac{2}{3\sqrt{q}}|\sin\frac{\theta_1+\theta_2}{2}|$, and so by (\ref{eq:eigenvalueinequalities})
$$
\la_i(\theta)\geq\la_6-\frac{2}{3\sqrt{q}}>-\la_1\qquad\textrm{for all $i=1,\ldots,6$ and all $\theta\in\RR^2$}.
$$
Hence $|\la_i(\theta)|=\la_1$ implies that $\la_i(\theta)=\la_1$. Then $\la_1(\theta)=\cdots=\la_i(\theta)$ and so if $i>1$ then $|\chi_{e^{i\theta}}(P^k)|>\chi_1(P^k)$ for sufficiently large $k$, contradicting Lemma~\ref{lem:bound1}. Therefore if $i>1$ then we have $|\la_i(\theta)|<\la_1$ for all $\theta\in\RR^2$. Finally we need to show that $|\la_1(\theta)|<\la_1$ unless $\theta_1$ and $\theta_2$ are multiples of $2\pi$. For this we observe the (rather remarkable) identity:
\begin{align*}
3\sqrt{q}\det(\pi_{e^{i\theta}}(P)-\la_1I)&=150-48(\cos\theta_1+\cos\theta_2+\cos(\theta_1+\theta_2))\\
&\qquad-2(\cos(\theta_1+2\theta_2)+\cos(2\theta_1+\theta_2)+\cos(\theta_1-\theta_2)),
\end{align*}
from which the result follows.

2. Write $\pi_{e^{i\varphi}}^{(1)}(P)=\pi_{1}^{(1)}(P)+E'(\varphi)$. Then the eigenvalues of $E'(\varphi)$  are $0$ and $\pm\frac{2}{3\sqrt{q}}|\sin\frac{\varphi}{2}|$. Therefore by (\ref{eq:eigenvalueinequalities}) we have $\mu_3(0)-\frac{2}{3\sqrt{q}}\leq\mu_i(\varphi)\leq \mu_1(0)+\frac{2}{3\sqrt{q}}$, and so
$$
\frac{1}{3\sqrt{q}}\big(q^{\frac{1}{2}}-2q^{-\frac{1}{2}}-4\big)\leq\mu_i(\varphi)\leq \frac{1}{3\sqrt{q}}\big(q^{\frac{1}{2}}+2q^{-\frac{1}{2}}+1\big)\quad\textrm{for each $i=1,2,3$.}
$$
It follows that $|\mu_i(\varphi)|<\la_1$ for all $i=1,2,3$ and all $\varphi\in\RR$.

3. This is obvious since $|\chi^{(2)}(P)|=q^{-\frac{3}{2}}$.
\end{proof}

\begin{lemma}\label{lem:c}
For $\theta_1,\theta_2\in\RR$ we have
$$
\frac{1}{|c(e^{i\theta})|^2}=\frac{q^6}{(q-1)^6}\theta_1^2\theta_2^2(\theta_1+\theta_2)^2\big(1+\cO(\|\theta\|^3)\big)
$$
\end{lemma}

\begin{proof}
Since $q>1$ we have
$$
\bigg|\frac{1-e^{-ix}}{1-q^{-1}e^{-ix}}\bigg|^2=\frac{q^2x^2}{(q-1)^2}\big(1+\cO(|x|^3)\big)\quad\textrm{for all $x\in\RR$}
$$
and the result follows from the definition of $c(e^{i\theta})$.
\end{proof}

\begin{lemma}\label{lem:combine} Let $w\in W$ and $n\in\ZZ_{\geq0}$. Then
$$
\chi_{e^{i\theta}}(P^nA_{w}^*)=C_w\,\la_1(\theta)^n\left(1+\cO(\|\theta\|)\right)+\fo(\la_1^n)\quad\textrm{where}\quad C_w=\vect{v}_1^T\pi_1(A_w^*)\vect{v}_1,
$$
where $\vect{v}_1=\vect{v}_1(0)$ is a unit eigenvector of $\pi_1(P)$ for $\la_1$.
\end{lemma}

\begin{proof} Let $X$ and $Y$ be $d\times d$ matrices with $X$ Hermitian. Let $X=PDP^T$ be an orthogonal diagonalisation with $D=\mathrm{diag}(\nu_1,\ldots,\nu_d)$ and $P=\begin{pmatrix}\vect{u}_1&\cdots&\vect{u}_d\end{pmatrix}$. Then
\begin{align*}
\tr(X^nY)&=\tr(PD^nP^TY)=\tr(D^nP^TYP)=\sum_{i=1}^d[P^TYP]_{i,i}\nu_i^n=\sum_{i=1}^d(\vect{u}_i^TY\vect{u}_i)\nu_i^n.
\end{align*}
Applying this to $X=\pi_{e^{i\theta}}(P)$ and $Y=\pi_{e^{i\theta}}(A_w^*)$ and using Lemma~\ref{lem:boundfull} gives
$$
\chi_{e^{i\theta}}(P^nA_w^*)=\left[\vect{v}_1(\theta)^T\pi_{e^{i\theta}}(A_w^*)\vect{v}_1(\theta)\right]\la_1(\theta)^n+\fo(\la_1^n).
$$
General perturbation theory gives $\vect{v}_1(\theta)=\vect{v}_1+\cO(\|\theta\|)$. Since the entries of the matrix $\pi_{e^{i\theta}}(A_w^*)$ satisfy $[\pi_{e^{i\theta}}(A_w^*)]_{ij}=[\pi_{1}(A_w^*)]_{ij}+\cO(\|\theta\|)$ (see Theorem~\ref{thm:combform}) it follows that 
\begin{align*}
\left[\vect{v}_1(\theta)^T\pi_{e^{i\theta}}(A_w^*)\vect{v}_1(\theta)\right]=\left[\vect{v}_1^T\pi_1(A_w^*)\vect{v}_1\right]\big(1+\cO(\|\theta\|)\big)&=C_w\big(1+\cO(\|\theta\|)\big).\qedhere
\end{align*}
\end{proof}

\begin{lemma}\label{lem:quadratic} We have
$$
\la_1(\theta)=\la_1\left(1-\beta\big(\theta_1^2+\theta_2^2+\theta_1\theta_2\big)+\cO(\|\theta\|^3)\right)\quad\textrm{where}\quad \beta=\frac{2}{9\la_1\sqrt{q^2+34q+1}}.
$$
\end{lemma}

\begin{proof} Since $\la_1$ has multiplicity~$1$, general results from perturbation theory imply that there is a neighbourhood of $(0,0)$ in which $\la_1(\theta)$ and $\vect{v}_1(\theta)$ are represented by convergent power series in the variables $\theta_1$ and $\theta_2$ (see \cite[Supplement, \S1]{baumgaertel}). The first few terms in these series can be computed in a few ways, for example by adapting the analysis of \cite[\S3.1.2]{baumgaertel} to the $2$-variable setting. The details are omitted. 
\end{proof}

\begin{thm} For the simple random walk on the chambers of a thick $\tilde{A}_2$ building with thickness $1<q<\infty$ we have
$$
p^{(n)}(c,d)=\frac{C_wq^{3-2\ell(w)}}{27\sqrt{3}\beta^4\pi(q-1)^6}\la_1^nn^{-4}\left(1+\cO\big(n^{-1/2}\big)\right)\qquad\textrm{if $\delta(c,d)=w$},
$$
where $\beta$ is as in Lemma~\ref{lem:quadratic} and $C_w$ is as in Lemma~\ref{lem:combine}.
\end{thm}

\begin{proof} By (\ref{eq:start}), Theorem~\ref{thm:plancherel}, and Lemma~\ref{lem:boundfull} we have
\begin{align*}
p^{(n)}(c,d)=\frac{q_w^{-2}}{6q^3(2\pi)^2}\int_{-\pi}^{\pi}\int_{-\pi}^{\pi}\frac{\chi_{e^{i\theta}}(P^nA_w^*)}{|c(e^{i\theta})|^2}\,d\theta_1 d\theta_2+\fo(\la_1^n)\qquad\textrm{if $\delta(c,d)=w$},
\end{align*}
and so using Lemma~\ref{lem:boundfull} again we have 
\begin{align}\label{eq:dr}
p^{(n)}(c,d)=\frac{1}{24\pi^2q^{2\ell(w)+3}}\int_{-\epsilon}^{\epsilon}\int_{-\epsilon}^{\epsilon}\frac{\chi_{e^{i\theta}}(P^nA_w^*)}{|c(e^{i\theta})|^2}\,d\theta_1d\theta_2+\fo(\la_1^n)
\end{align}
for small $\epsilon>0$. Let $I_n$ be the double integral in (\ref{eq:dr}). Let $\varphi_1=\sqrt{n}\theta_1$ and $\varphi_2=\sqrt{n}\theta_2$. By Lemma~\ref{lem:c} we have
$$
\frac{1}{|c(e^{i\varphi/\sqrt{n}})|^2}=\frac{q^6}{(q-1)^6}g(\varphi)n^{-3}\big(1+\cO(n^{-1})\big),\quad\textrm{where}\quad g(\varphi)=\varphi_1^2\varphi_2^2(\varphi_1+\varphi_2)^2
$$
and so
$$
I_n=\frac{q^6}{(q-1)^6}n^{-4}\big(1+\cO(n^{-1})\big)\int_{-\sqrt{n}\epsilon}^{\sqrt{n}\epsilon}\int_{-\sqrt{n}\epsilon}^{\sqrt{n}\epsilon} g(\varphi)\chi_{e^{i\varphi/\sqrt{n}}}(P^nA_w^*)\,d\varphi_1d\varphi_2.
$$
By Lemma~\ref{lem:combine} we have
\begin{align*}
\chi_{e^{i\varphi/\sqrt{n}}}(P^nA_w^*)&=C_w\la_1(\varphi/\sqrt{n})^n\left(1+\cO(n^{-\frac{1}{2}})\right)+\fo(\la_1^n).
\end{align*}
Writing $h(\varphi)=\varphi_1^2+\varphi_1\varphi_2+\varphi_2^2$, Lemma~\ref{lem:quadratic} gives
\begin{align*}
\la_1(\varphi/\sqrt{n})^n&=\la_1^n\left(1-\beta h(\varphi)n^{-1}+\cO(n^{-3/2})\right)^n\\
&=\la_1^n\left(e^{-\beta h(\varphi)/n}+\cO\big(n^{-3/2}\big)\right)^n=\la_1^ne^{-\beta h(\varphi)}\left(1+\cO\big(n^{-1/2}\big)\right).
\end{align*}
Therefore
$$
I_n=C_w\frac{q^6}{(q-1)^6}\la_1^nn^{-4}\left(1+\cO\big(n^{-1/2}\big)\right)\int_{-\sqrt{n}\epsilon}^{\sqrt{n}\epsilon}\int_{-\sqrt{n}\epsilon}^{\sqrt{n}\epsilon} g(\varphi)e^{-\beta h(\varphi)}d\varphi_1d\varphi_2.
$$
The integral tends to
$$
\int_{-\infty}^{\infty}\int_{-\infty}^{\infty}g(\varphi)e^{-\beta h(\varphi)}\,d\varphi_1d\varphi_2=\frac{1}{\beta^4}\int_{-\infty}^{\infty}\int_{-\infty}^{\infty}g(\varphi)e^{-h(\varphi)}\,d\varphi_1d\varphi_2=\frac{8\pi}{9\sqrt{3}\beta^4},
$$
and the result follows from (\ref{eq:dr}).
\end{proof}

\begin{remark} The spectral radius formula $\la_1=\frac{3(q-1)+\sqrt{q^2+34q+1}}{6q}$ agrees with computations made by Saloff-Coste and Woess in \cite[Example~6]{woesslaurent}.
\end{remark}

\begin{remark}
Let $G=SL_3(\FF)$ where $\FF$ is a non-archimedean local field. Let $I$ be the \textit{standard Iwahori subgroup} of $G$, defined by the following diagram, where $\fo$ is the ring of integers in $\FF$ and where $\theta:\FF\to k$ is the canonical homomorphism onto the residue field $k$ (for example, $\FF=\FF_q((t))$, $\fo=\FF_q[[t]]$, $k=\FF_q$ and $\theta=\mathrm{ev}_{t=0}$).
\begin{align*}
\begin{matrix}
G &= &SL_3(\FF) \\
$\beginpicture
\setcoordinatesystem units <0.8cm,0.8cm>         
\setplotarea x from -0.1 to 0.3, y from -0.5 to 0.5  
\put{$\cup$} at 0 0 \put{$\shortmid$} at 0.21 0.08
\put{$\shortmid$} at 0.21 0.0
\endpicture$ &&$\beginpicture
\setcoordinatesystem units <0.8cm,0.8cm>         
\setplotarea x from -0.1 to 0.3, y from -0.5 to 0.5  
\put{$\cup$} at 0 0 \put{$\shortmid$} at 0.21 0.08
\put{$\shortmid$} at 0.21 0.0
\endpicture$ \\
K &=& SL_3(\fo) &\mapright{\theta} &SL_3(k) \\
$\beginpicture
\setcoordinatesystem units <0.8cm,0.8cm>         
\setplotarea x from -0.1 to 0.3, y from -0.5 to 0.5  
\put{$\cup$} at 0 0 \put{$\shortmid$} at 0.21 0.08
\put{$\shortmid$} at 0.21 0.0
\endpicture$ &&$\beginpicture
\setcoordinatesystem units <0.8cm,0.8cm>         
\setplotarea x from -0.1 to 0.3, y from -0.5 to 0.5  
\put{$\cup$} at 0 0 \put{$\shortmid$} at 0.21 0.08
\put{$\shortmid$} at 0.21 0.0
\endpicture$ &&$\beginpicture
\setcoordinatesystem units <0.8cm,0.8cm>         
\setplotarea x from -0.1 to 0.3, y from -0.5 to 0.5  
\put{$\cup$} at 0 0 \put{$\shortmid$} at 0.21 0.08
\put{$\shortmid$} at 0.21 0.0
\endpicture$ \\
I &= &\theta^{-1}(B(k)) &\mapright{\theta} &B(k)
\end{matrix}
\end{align*}
where $B(k)$ is the subgroup of upper triangular matrices in $SL_3(k)$. Then $G/I$ is the set of chambers of an $\tilde{A}_2$ building (and $G/K$ is the set of type~$0$ vertices of that building). Since $\cC_w(gI)=(gIwI)/I$ our local limit theorem gives a local limit theorem for bi-$I$-invariant probability measures on $G$.
\end{remark}

\section*{Part II: Harmonic analysis on affine Hecke algebras}

In this part we give an outline of some well known structural theory of affine Hecke algebras. We prove Opdam's generating function formula for the trace functional. Our argument is slightly different to Opdam's~\cite{opdamtrace} (we prove the formula by applying the harmonic analysis on the centre of the Hecke algebra). This formula is at the heart of harmonic analysis on affine Hecke algebras. We apply it to prove the Plancherel Theorem for type $\tilde{A}_2$, following the general technique of~\cite{opdamhanalysis}.

\section{Affine Weyl groups and alcove walks}\label{sect:4}

In this section we fix some standard notation on affine Weyl groups, and briefly discuss the combinatorics of \textit{alcove walks}. Alcove walks control many aspects of the representation theory of Lie algebras and Hecke algebras. Standard references for this section include \cite{bourbaki}, \cite{humphreys} and \cite{ramalcove}.

\subsection{Root systems and affine Weyl groups}

Let us fix some notation, mainly following \cite{bourbaki}.
\begin{enumerate}
\item[$\bullet$] Let $\fh$ be an $n$-dimensional real vector space with inner product $\langle\cdot,\cdot\rangle$.
\item[$\bullet$] For nonzero $\alpha\in \fh$ let $\alpha^{\vee}=2\alpha/\langle\alpha,\alpha\rangle$.
\item[$\bullet$] Let $R$ be a \textit{reduced irreducible root system} in~$\fh$ (see \cite{bourbaki} for the classification).
\item[$\bullet$] Let $\{\alpha_1,\ldots,\alpha_n\}$ be a set of \textit{simple roots} of $R$.
\item[$\bullet$] Let $R^+$ be the set of \textit{positive roots}. Let $\varphi\in R$ be the \textit{highest root}.
\item[$\bullet$] For $\alpha\in R$ let $H_{\alpha}=\{\la\in \fh\mid\langle\la,\alpha\rangle=0\}$ be the hyperplane orthogonal to~$\alpha$.
\item[$\bullet$] For $\alpha\in R$ let $s_{\alpha}\in GL(\fh)$ be the reflection $s_{\alpha}(\la)=\la-\langle\la,\alpha\rangle\alpha^{\vee}$ through~$H_{\alpha}$.
\item[$\bullet$] Let $Q=\ZZ\alpha_1^{\vee}+\cdots+\ZZ\alpha_n^{\vee}$ be the \textit{coroot lattice}, and $Q^+=\ZZ_{\geq0}\alpha_1^{\vee}+\cdots+\ZZ_{\geq0}\alpha_n^{\vee}$.
\item[$\bullet$] Let $\{\omega_1,\ldots,\omega_n\}$ be the dual basis to $\{\alpha_1,\ldots,\alpha_n\}$ defined by $\langle\omega_i,\alpha_j\rangle=\delta_{ij}$.
\item[$\bullet$] Let $P=\ZZ\omega_1+\cdots+\ZZ\omega_n$ be the \textit{coweight lattice}, and $P^+=\ZZ_{\geq0}\omega_1+\cdots+\ZZ_{\geq0}\omega_n$ be the cone of \textit{dominant coweights}.
\end{enumerate}
In Figure~\ref{fig:A_2} the lattice $Q$ consists of the centres of the solid hexagons, and the latttice $P$ consist of all vertices in the picture. 

The \textit{Weyl group} $W_0$ of $R$ is the subgroup of $GL(\fh)$ generated by $\{s_{\alpha}\mid\alpha\in R\}$. The Weyl group is a finite Coxeter group with distinguished generators $s_1,\ldots,s_n$ (where $s_i=s_{\alpha_i}$) and thus has a length function $\ell:W_0\to\ZZ_{\geq0}$, with $\ell(w)$ being the smallest $\ell\geq0$ such that $w=s_{i_1}\cdots s_{i_{\ell}}$. Let $w_0$ be the (unique) longest element of~$W_0$. The \textit{inversion set} of $w\in W_0$ is
$$
R(w)=\{\alpha\in R^+\mid w^{-1}\alpha\in-R^+\},\qquad\textrm{and}\qquad  \ell(w)=|R(w)|.
$$

The open connected components of $\fh\backslash\bigcup_{\alpha\in R}H_{\alpha}$ are \textit{Weyl sectors}. These are open simplicial cones, and $W_0$ acts simply transitively on the set of Weyl sectors. The \textit{fundamental Weyl sector} is
$
S_0=\{\la\in\fh\mid\langle\la,\alpha_i\rangle>0\textrm{ for }i=1,\ldots,n\},
$
and $P^+=P\cap\overline{S_0}$, where $\overline{S_0}$ is the closure of $S_0$ in $\fh$.

The roots $\alpha\in R$ can be regarded as elements of $\fh^*$ by setting $\alpha(\la)=\langle\la,\alpha\rangle$ for~$\la\in\fh$. Let $\delta:\fh\to\RR$ be the (non-linear) constant function with $\delta(\la)=1$ for all $\la\in\fh$. The \textit{affine root system} is $R_{\mathrm{aff}}=R+\ZZ\delta$. The \textit{affine hyperplane} for the affine root $\alpha+j\delta$ is
$$
H_{\alpha+j\delta}=\{\la\in\fh\mid\langle\la,\alpha+j\delta\rangle=0\}=\{\la\in\fh\mid\langle\la,\alpha\rangle=-j\}=H_{-\alpha-j\delta}.
$$
The \textit{affine Weyl group} is the subgroup $W$ of $\mathrm{Aff}(\fh)$ generated by the reflections $s_{\alpha+k\delta}$ with $\alpha+k\delta\in R_{\mathrm{aff}}$, where the reflection $s_{\alpha+k\delta}:\fh\to\fh$ is given by the formula
$
s_{\alpha+k\delta}(\la)=\la-(\langle\la,\alpha\rangle+k)\alpha^{\vee}
$ for $\la\in\fh$.
Let $\alpha_0=-\varphi+\delta$ (with $\varphi$ the highest root of $R$). The affine Weyl group is a Coxeter group with distinguished generators $s_0,s_1,\ldots,s_n$, where $s_0=s_{\alpha_0}$. For $\mu\in\fh$, let $t_{\mu}:\fh\to\fh$ be the translation $t_{\mu}(\la)=\la+\mu$ for all $\la\in\fh$. Then $s_{\alpha+k\delta}=t_{-k\alpha^{\vee}}s_{\alpha}$ and $W$ is the semidirect product $W=Q\rtimes W_0$.

The open connected components of $\fh\backslash\bigcup_{\beta\in R_{\mathrm{aff}}}H_{\beta}$ are \textit{chambers} (or \textit{alcoves}). The \textit{fundamental chamber} is
$$
c_0=\{\la\in\fh\mid\langle\la,\alpha_i\rangle>0\textrm{ for all $i=0,\ldots,n$}\}\subset S_0.
$$
The affine Weyl group acts simply transitively on the set of chambers, and therefore $W$ is in bijection with the set of chambers. Identify $1$ with $c_0$. 

The \textit{extended affine Weyl group} $\tilde{W}=P\rtimes W_0$ acts transitively (but in general not simply transitively) on the set of chambers. In general $\tilde{W}$ is not a Coxeter group, but it is ``nearly'' a Coxeter group: There is a length function $\ell:\tilde{W}\to\ZZ_{\geq0}$ defined by 
$
\ell(w)=|\{H_{\alpha+j\delta}\mid \textrm{$H_{\alpha+j\delta}$ separates $c_0$ from $wc_0$}\}|,
$
and for $w\in W\subseteq \tilde{W}$ this agrees with the Coxeter length function. Let $\Gamma=\{w\in\tilde{W}\mid\ell(w)=0\}$. Then $\tilde{W}=W\rtimes\Gamma$, and $\Gamma$ is isomorphic to the finite abelian group~$P/Q$. Therefore $\tilde{W}$ acts simply transitively on the set of chambers in $\fh\times\Gamma$, and so $\tilde{W}$ can be thought of as $|\Gamma|$ copies of $W$.

If $w\in \tilde{W}=P\rtimes W_0$ we define the \textit{weight} $\wt(w)\in P$ and the \textit{final direction} $\theta(w)\in W_0$ by the equation
\begin{align}\label{eq:weightdirection}
w=t_{\wt(w)}\theta(w).
\end{align}

The \textit{Bruhat partial order} on $W$ is defined as follows: $v\leq w$ if and only if $v$ is a `subexpression' of a reduced expression $w=s_{i_1}\cdots s_{i_{\ell}}$ for $w$. Here subexpression means an expression obtained by deleting one or more factors from the expression $w=s_{i_1}\cdots s_{i_{\ell}}$. If $v\leq w$ then $v$ is a subexpression of \textit{every} reduced expression for~$w$. The Bruhat order extends to $\tilde{W}$ by setting $v\leq w$ if and only if $w=w'\gamma$ and $v=v'\gamma$ with $w',v'\in W$ and $\gamma\in\Gamma$ and $v'\leq w'$.

\subsection{Alcove Walks}

Each affine hyperplane $H_{\alpha+k\delta}$ determines two closed halfspaces of $\fh$. Define an \textit{orientation} on the affine hyperplane $H_{\alpha+k\delta}$ by declaring the positive side to be the half space which contains a subsector of the fundamental sector~$S_0$. Explicitly, if $\alpha\in R^+$ and $k\in\ZZ$ then the negative and positive sides of $H_{\alpha+k\delta}$ are
\begin{align*}
H_{\alpha+k\delta}^-&=\{x\in\fh\mid\langle x,\alpha+k\delta\rangle\leq0\}=\{x\in\fh\mid \langle x,\alpha\rangle\leq -k\},\\
H_{\alpha+k\delta}^+&=\{x\in\fh\mid\langle x,\alpha+k\delta\rangle\geq0\}=\{x\in\fh\mid\langle x,\alpha\rangle\geq -k\}.
\end{align*}
See the picture in Example~\ref{ex:A_2}; note that this orientation is translation invariant.

Let $\vec w=s_{i_1}\cdots s_{i_{\ell}}\gamma$ be an expression for $w\in \tilde{W}$, with $\gamma\in\Gamma$. A \textit{positively folded alcove walk of type~$\vec w$} is a sequence of steps from alcove to alcove in $\tilde{W}$, starting at $1\in \tilde{W}$, and made up of the symbols
\begin{align}\label{eq:symbols}
\beginpicture
\setcoordinatesystem units <0.8cm,0.8cm>         
\setplotarea x from -0.8 to 0.7, y from -0.5 to 0.5  
\put{$\scriptstyle{-}$}[b] at -0.4 0.25
\put{$\scriptstyle{+}$}[b] at 0.4 0.25
\put{$\scriptstyle{x}$}[br] at -0.6 0.1
\put{$\scriptstyle{xs_i}$}[bl] at 0.6 0.1
\plot  0 -0.4  0 0.5 /
\arrow <5pt> [.2,.67] from -0.5 0 to 0.5 0   %
\put{(\textit{positive $i$-crossing})} at 0 -1
\endpicture
\qquad\qquad
\beginpicture
\setcoordinatesystem units <0.8cm,0.8cm>         
\setplotarea x from -0.8 to 0.7, y from -0.5 to 0.5  
\put{$\scriptstyle{-}$}[b] at -0.4 0.35
\put{$\scriptstyle{+}$}[b] at 0.4 0.35
\put{$\scriptstyle{x}$}[bl] at 0.6 0.1
\put{$\scriptstyle{xs_i}$}[br] at -0.6 0.1
\plot  0 -0.4  0 0.6 /
\plot 0.5 0  0.05 0 /
\arrow <5pt> [.2,.67] from 0.05 0.1 to 0.5 0.1   %
\plot 0.05 0 0.05 0.1 /
\put{(\textit{positive $i$-fold})} at 0 -1
\endpicture
\qquad\qquad
\beginpicture
\setcoordinatesystem units <0.8cm,0.8cm>         
\setplotarea x from -0.8 to 0.7, y from -0.5 to 0.5  
\put{$\scriptstyle{-}$}[b] at -0.4 0.25
\put{$\scriptstyle{+}$}[b] at 0.4 0.25
\put{$\scriptstyle{x}$}[bl] at 0.6 0.1
\put{$\scriptstyle{xs_i}$}[br] at -0.6 0.1
\plot  0 -0.4  0 0.5 /
\arrow <5pt> [.2,.67] from 0.5 0 to -0.5 0   %
\put{(\textit{negative $i$-crossing})} at 0 -1 
\endpicture
\end{align}
where the $k$th step has $i={i_k}$ for $k=1,\ldots,\ell$. To take into account the sheets of $\tilde{W}$, one concludes the alcove walk by ``jumping'' to the $\gamma$ sheet of $\fh\times\Gamma$. Our pictures will always be drawn without this jump by projecting $\fh\times\Gamma\rightarrow\fh\times \{1\}$.

Let $p$ be a positively folded alcove walk. For each $i=0,1,\ldots,n$ let
\begin{align*}
f_i(p)&=\#\textrm{(type $i$-folds in $p$)}.
\end{align*}

Let $\vec{w}=s_{i_1}\cdots s_{i_{\ell}}\gamma$ be a reduced expression for $w\in \tilde{W}$. Define
\begin{align}\label{eq:cP}
\cP(\vec w)=\{\textrm{all positively folded alcove walks of type $\vec w$}\}.
\end{align}
Let $\mathrm{end}(p)\in\tilde{W}$ be the alcove where $p$ ends. By the definition of the Bruhat order it is clear that if $p\in\cP(\vec w)$ (with $\vec w$ reduced) then 
\begin{align}\label{eq:endbruhat}
\mathrm{end}(p)\leq w\qquad\textrm{in Bruhat order}.
\end{align}
Define the \textit{weight} $\wt(p)\in P$ and \textit{final direction} $\theta(p)\in W_0$ by the equation
\begin{align}\label{eq:finaleq}
\mathrm{end}(p)=t_{\wt(p)}\theta(p).
\end{align}
The \textit{dominance order} on $P$ is given by $\mu\preceq \la$ if and only if $\la-\mu\in Q^+$. It is not difficult to show that if $p\in\cP(\vec w)$ then
\begin{align}\label{eq:enddominance}
\wt(w)\preceq\wt(p).
\end{align}
This is a consequence of the paths being `positively' folded.

\begin{example} The positively folded alcove walk $p$ 
 \begin{figure}[h]
 \begin{center}
       \includegraphics[totalheight=4.5cm]{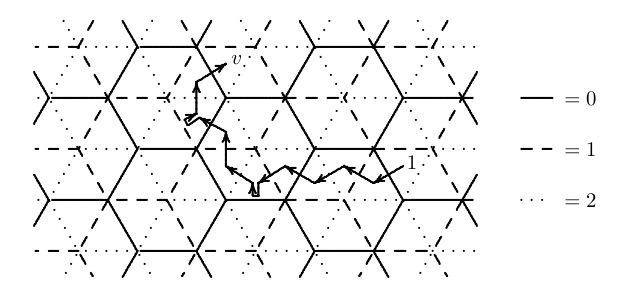}
 \end{center}
 \caption{A positively folded alcove walk in type $\tilde{A}_2$}\label{fig:walk1}
 \end{figure}

\noindent has type $\vec{w}=s_0s_1s_2s_0s_1s_0s_2s_1s_0s_1s_2s_0$ (this is reduced). The end chamber of $p$ is 
$\mathrm{end}(p)=v=s_0s_1s_2s_0s_1s_2s_1s_0s_2s_0=s_0s_1s_2s_0s_2s_1s_0s_2\leq w,$
and $f_0(p)=1$, $f_1(p)=1$, and $f_2(p)=0$. We have $\wt(p)=4\omega_1-\omega_2$ and $\wt(w)=5\omega_1-6\omega_2$. Note that $\wt(p)-\wt(w)=\alpha_1^{\vee}+3\alpha_2^{\vee}\in Q^+$ and so $\wt(w)\preceq\wt(p)$.
\end{example}

\begin{example} In type $\tilde{A}_2$, the set $\cP(s_1s_2s_1s_0)$ consists of the $10$ paths in Figure~\ref{fig:walks2} (arranged according to $\wt(p)$).

 \begin{figure}[h]
 \begin{center}
       \includegraphics[totalheight=4.5cm]{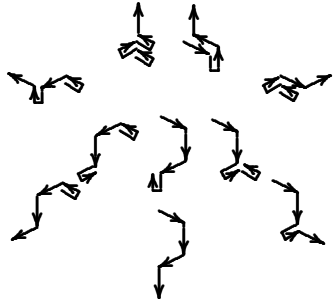}
 \end{center}
 \caption{Positively folded alcove walks of type $\vec w=s_1s_2s_1s_0$}\label{fig:walks2}
 \end{figure} 
 
\noindent The bottom path is $\vec{w}$ (the path with no folds). Note that all other paths have $\mathrm{end}(p)\leq w$ and $\wt(\vec{w})\preceq\wt(p)$.
\end{example}

\subsection{Parameter systems}

Let $R,W_0,W,\tilde{W}$, etc be as above. A \textit{parameter system} is a set $\mathbf{q}=\{q_0,q_1,\ldots,q_n\}$ such that
(i) $q_i>1$ for each $i=0,1,\ldots,n$, and (ii) $q_i=q_j$ whenever $s_i$ and $s_j$ are conjugate in $W$. For example, the parameters of a locally finite regular building form a parameter system. A parameter system is \textit{reduced} if it satisfies (iii) if $R$ is of type $A_1$ then $q_0=q_1$, and if $R$ is of type $C_n$ then $q_0=q_n$. The `reduced' hypothesis can be removed, but without it some of the subsequent formulae become more complex.

Let $\mathbf{q}$ be a reduced parameter system. By \cite[IV, \S1, No.5, Prop~5]{bourbaki} 
$$
q_w:=q_{i_1}\cdots q_{i_{\ell}}\qquad\textrm{if $w=s_{i_1}\cdots s_{i_{\ell}}\in W$ is a reduced expression}
$$
does not depend on the choice of reduced expression. Extend this definition to $\tilde{W}$ be setting
$q_{w\gamma}=q_w$ whenever $w\in W$ and $\gamma\in\Gamma$. For $\alpha\in R$, define $q_{\alpha}$ by
$$
q_{\alpha}=q_i\qquad\textrm{if $\alpha\in W_0\alpha_i$}.
$$
Since $\alpha\in W_0\alpha_i\cup W_0\alpha_j$ implies that $s_j=ws_iw^{-1}$ for some $w\in W_0$ this definition is unambiguous. 

\subsection{Extended affine Hecke algebras}

\begin{defn} Let $\tilde{W}$ be an extended affine Weyl group and let $\mathbf{q}$ be a reduced parameter system. The \textit{extended affine Hecke algebra with Weyl group $\tilde{W}$ and parameter system $\mathbf{q}$} is the algebra $\scH$ over $\CC$ with generators 
$
T_w
$ ($w\in \tilde{W}$) and defining relations
\begin{align*}
T_uT_v&=T_{uv}&&\textrm{if $\ell(uv)=\ell(u)+\ell(v)$}\\
T_wT_{s_i}&=T_{ws_i}+(q_i^{\frac{1}{2}}-q_i^{-\frac{1}{2}})T_w&&\textrm{if $\ell(ws_i)=\ell(w)-1$}.
\end{align*}
We will usually drop the adjective `extended' and call $\scH$ the affine Hecke algebra.
\end{defn}

\begin{remark} We often write $T_i$ in place of $T_{s_i}$ for $i=0,1,\ldots,n$. One immediately sees that each $T_i$ is invertible, with inverse $T_i^{-1}=T_i-(q_i^{\frac{1}{2}}-q_i^{-\frac{1}{2}})$, and that $T_{\gamma}^{-1}=T_{\gamma^{-1}}$ for $\gamma\in \Gamma$. It follows that each $T_w$, $w\in \tilde{W}$, is invertible. 
\end{remark}

\begin{remark}\label{rem:isomorphism}
If $q_0,q_1,\ldots,q_n$ are the parameters of a locally finite regular building then the subalgebra $\scH_{W}$ of $\scH$ generated by $T_w$, $w\in W$, is isomorphic to $\scA$, with $T_w\mapsto q_w^{1/2}A_w$ (see Proposition~\ref{prop:relations}). This renormalisation leads to neater formulas in the Hecke algebra theory. Also it is more convenient to work in the larger extended Hecke algebra.
\end{remark}

\section{Structure of affine Hecke algebras}\label{sect:5}

This section is classical and well known to experts. Standard references include \cite{lusztig}, \cite{macblue}, \cite{ram1}, and \cite{xi}. The main results we describe are:
\begin{enumerate}
\item[$\bullet$] The \textit{Bernstein presentation}. This realises the semidirect product structure $\tilde{W}=P\rtimes W_0$ of the extended affine Weyl group at the Hecke algebra level.
\item[$\bullet$] The computation of the centre of $\scH$. This is useful because the centre of an algebra plays an important role in its representation theory.
\item[$\bullet$] The derivation of the \textit{Macdonald formula}. This formula is key to the Plancherel formula on the centre of~$\scH$. 
\end{enumerate}

\subsection{Bernstein presentation of $\scH$}

Let $v\in \tilde{W}$, and choose \textit{any} expression $v=s_{i_1}\cdots s_{i_{\ell}}\gamma$ for $v$ (not necessarily reduced). Interpret this expression as an alcove walk with no folds starting at the alcove $1\in \tilde{W}$. Let $\epsilon_1,\ldots,\epsilon_{\ell}\in\{-1,+1\}$ be the signs of the crossings of this walk. The element
$$
x_{v}=T_{i_1}^{\epsilon_1}\cdots T_{i_{\ell}}^{\epsilon_{\ell}}T_{\gamma}
$$
does not depend on the particular expression for $v$ chosen (see \cite{goertz}).

\begin{prop}\label{prop:pathformula} Let $w\in\tilde{W}$, and choose a reduced expression $\vec{w}=s_{i_1}\cdots s_{i_{\ell}}\gamma$. Then
$$
T_w=\sum_{p\in\cP(\vec{w})}\cQ(p)x_{\mathrm{end}(p)}\qquad\textrm{where}\qquad\cQ(p)=\prod_{i=0}^n(q_i^{\frac{1}{2}}-q_i^{-\frac{1}{2}})^{f_i(p)}.
$$
\end{prop}

\begin{proof} This is an easy induction using the formula $T_i=T_i^{-1}+(q_i^{\frac{1}{2}}-q_i^{-\frac{1}{2}})$.
\end{proof}

\begin{cor}\label{cor:basis} The set $\{x_v\mid v\in \tilde{W}\}$ is a basis of $\scH$. The transition matrices converting between the bases $\{T_w\mid w\in \tilde{W}\}$ and $\{x_v\mid v\in \tilde{W}\}$ are upper triangular with respect to the Bruhat order, and have $1$s on the main diagonal.
\end{cor}

For $\mu\in P$, define 
$
x^{\mu}=x_{t_{\mu}}.
$

The relations in the following presentation of $\scH$ are the algebra analogues of the defining relations: 
$$
s_i^2=1,\qquad \underbrace{s_is_js_i\cdots}_{\textrm{$m_{ij}$ terms}}=\underbrace{s_js_is_j\cdots}_{\textrm{$m_{ij}$ terms}},\qquad t_{\la}t_{\mu}=t_{\la+\mu}=t_{\mu+\la},\qquad s_it_{\la}=t_{s_i\la}s_i.
$$
($i,j=1,\ldots,n$ and $\la,\mu\in P$) in the extended affine Weyl group $\tilde{W}$.

\begin{thm}[Bernstein Presentation]\label{thm:pres} For all $i,j=1,\ldots,n$ and all $\la,\mu\in P$ we have
    \begin{align*}
    T_i^2&=1+(q_i^{\frac{1}{2}}-q_i^{-\frac{1}{2}})T_i\\
    T_iT_jT_i\cdots&=T_jT_iT_j\cdots&&\textrm{($m_{ij}$ terms on each side)}\\
    x^{\la}x^{\mu}&=x^{\la+\mu}=x^{\mu}x^{\la}\\
    T_{i}x^{\mu}
    &=x^{s_i\mu}T_{i}+(q_i^{\frac{1}{2}}-q_i^{-\frac{1}{2}})
    \frac{x^{\mu}-x^{s_i\mu}}{1-x^{-\alpha_i^{\vee}}}&&\textrm{(the Bernstein relation)}.
    \end{align*}
\end{thm}

\begin{proof} These facts can be deduced from the alcove walk setup. See~\cite{ramalcove}.
\end{proof}

\begin{remark} The `fraction' appearing in the Bernstein relation is actually an element of $\CC[P]$, because $s_i\mu=\mu-\langle\mu,\alpha_i\rangle\alpha_i^{\vee}$, and $\langle\mu,\alpha_i\rangle\in\ZZ$ since $\mu\in P$.
\end{remark}

\begin{cor}\label{cor:basis2} The sets 
$$
\{x^{\mu}T_{w}\mid \mu\in P,w\in W_0\}\qquad\textrm{and}\qquad\{T_wx^{\mu}\mid\mu\in P,w\in W_0\}
$$
are both bases for $\scH$.
\end{cor}

\begin{proof}
Since $t_{\mu}$ is in the `$1$-position' of
$t_{\mu}W_0$, and since the orientation on the hyperplanes is translation invariant we have
$$
x_{t_{\mu}w}=x^{\mu}T_{w^{-1}}^{-1}\qquad\textrm{for
all $\mu\in P$ and all $w\in W_0$.}
$$
Therefore $\{x^{\mu}T_{w^{-1}}^{-1}\mid \mu\in P, w\in W_0\}$
is a basis of $\scH$, and the result follows from Corollary~\ref{cor:basis} and the Bernstein relation.
\end{proof}

It is not difficult to use the Bernstein relation to compute the centre of $\scH$. Let $\CC[P]$ denote the $\CC$-span of the elements $x^{\la}$, $\la\in P$. Then $\CC[P]$ carries a natural $W_0$-action (with $w\cdot x^{\la}=x^{w\la}$), and we write
$$
\CC[P]^{W_0}=\{p\in\CC[P]^{W_0}\mid w\cdot p=p\textrm{ for all $w\in W_0$}\}.
$$

\begin{cor}\label{cor:centre} The centre of $\scH$ is $Z(\scH)=\CC[P]^{W_0}$.
\end{cor}

\begin{proof}
If $z\in\CC[P]^{W_0}$ then Theorem~\ref{thm:pres} gives 
$T_wz=zT_w$ and $x^{\mu}z=zx^{\mu}$ for all $w\in W_0$ and $\mu\in P$. Therefore $z\in Z(\scH)$. Conversely suppose that $z\in Z(\scH)$. Use Corollary~\ref{cor:basis2} to write
$$
z=\sum_{w\in W_0}p_w(x)T_w\qquad\textrm{where $p_w(x)\in\CC[P]$}.
$$
Let $w$ be a maximal element of $W_0$ (in the Bruhat order) subject to the condition that $p_w(x)\neq 0$. Since $x^{\la}zx^{-\la}=z$ the Bernstein relation gives
$
p_w(x)=x^{\la-w\la}p_w(x)$ for all $\la\in P$, and so $w=1$. Therefore $z\in\CC[P]$. Then for $i=1,\ldots,n$ we have
$$
zT_i=T_iz=(s_iz)T_i+z'\qquad\textrm{for some $z'\in\CC[P]$},
$$
and so $zT_i=(s_iz)T_i$ by Corollary~\ref{cor:basis2}. Thus $z=s_iz$ for each $i$, so $z\in \CC[P]^{W_0}$.
\end{proof}

\subsection{The Macdonald formula}

It is natural to seek modifications $\tau_w$ of the elements $T_w$ which satisfy the ``simplified Bernstein relation''
$$
\tau_wx^{\mu}=x^{w\mu}\tau_w\qquad\textrm{for all $w\in W_0$ and $\mu\in P$}.
$$
For each $i=1,\ldots,n$ define the \textit{intertwiner} $\tau_i\in\scH$ by
$$
\tau_i=(1-x^{-\alpha_i^{\vee}})T_i-(q_i^{\frac{1}{2}}-q_i^{-\frac{1}{2}}).
$$
By Theorem~\ref{thm:pres} we have
$
\tau_ix^{\mu}=x^{s_i\mu}\tau_i
$ for all $\mu\in P$, and a direct computation (using Theorem~\ref{thm:pres}) gives
\begin{align}\label{eq:iii}
\tau_i^2=q_i(1-q_i^{-1}x^{-\alpha_i^{\vee}})(1-q_i^{-1}x^{\alpha_i^{\vee}})\in\CC[P].
\end{align}
It can be shown that 
$$
\tau_w=\tau_{i_1}\cdots\tau_{i_{\ell}}
$$
is independent of the choice of reduced expression $w=s_{i_1}\cdots s_{i_{\ell}}\in W_0$, and that the $\tau_w$ are linearly independent over $\CC[P]$.

Define $\mathbf{1}_0\in\scH$ by
\begin{align}\label{eq:inverses}
    \mathbf{1}_0=\frac{1}{W_0(q)}\sum_{w\in
    W_0}q_w^{\frac{1}{2}}T_w,\quad\textrm{where}\quad W_0(q)=\sum_{w\in W_0}q_w.
\end{align}
Induction on $\ell(w)$ shows that
\begin{align}\label{eq:1}
T_w\mathbf{1}_0=\mathbf{1}_0T_w=q_w^{\frac{1}{2}}\mathbf{1}_0\quad\textrm{for
all $w\in W_0$},\quad\textrm{and so}\quad \mathbf{1}_0^2=\mathbf{1}_0.
\end{align}
Therefore
\begin{align}\label{eq:later}
\mathbf{1}_0\tau_i=q_i^{\frac{1}{2}}\mathbf{1}_0(1-q_i^{-1}x^{\alpha_i^{\vee}})\quad\textrm{and}\quad\tau_i\mathbf{1}_0=-q_i^{-\frac{1}{2}}x^{-\alpha_i^{\vee}}(1-q_i^{-1}x^{\alpha_i^{\vee}})\mathbf{1}_0.
\end{align}

Define elements $d(x),n(x)\in\CC[P]$ by
\begin{align}
\label{eq:DN}d(x)=\prod_{\alpha\in R^+}(1-x^{-\alpha^{\vee}})\qquad\textrm{and}\qquad n(x)=\prod_{\alpha\in R^+}(1-q_{\alpha}^{-1}x^{-\alpha^{\vee}}).
\end{align}

\begin{thm}\label{thm:1_0} We have
$$
d(x)\mathbf{1}_0=\frac{q_{w_0}}{W_0(q)}\sum_{w\in W_0}q_w^{-\frac{1}{2}}c_w\tau_w,\quad\textrm{where}\quad c_w=\prod_{\alpha\in R(w^{-1}w_0)}(1-q_{\alpha}^{-1}x^{-w\alpha^{\vee}}).
$$
In particular, the coefficient of $\tau_e$ in $d(x)\mathbf{1}_0$ is $\frac{q_{w_0}}{W_0(q)}n(x)$.
\end{thm}

\begin{proof} We have
$
d(x)\mathbf{1}_0=\sum_{w\in W_0}a_w\tau_w
$ for some polynomials $a_w\in\CC[P]$ (because each $d(x)T_w$ with $w\in W_0$ has this property). This expression is unique, because the $\tau_w$ are linearly independent over $\CC[P]$, and obviously $a_{w_0}=q_{w_0}^{1/2}W_0(q)^{-1}$. On the one hand using (\ref{eq:later}) we see that for each $i=1,\ldots,n$ we have
$$
d(x)\mathbf{1}_0\tau_i=d(x)\mathbf{1}_0q_i^{\frac{1}{2}}(1-q_i^{-1}x^{\alpha_i^{\vee}})=\sum_{w\in W_0}q_i^{\frac{1}{2}}a_w(1-q_i^{-1}x^{w\alpha_i^{\vee}})\tau_w,
$$
and on the other hand direct computation gives
\begin{align*}
d(x)\mathbf{1}_0\tau_i&=\sum_{w\in W_0}a_w\tau_w\tau_i=\sum_{w:ws_i<w}a_{ws_i}\tau_w+\sum_{w:ws_i>w}a_{ws_i}\tau_w\tau_i^2.
\end{align*}
Since $\tau_i^2\in\CC[P]$ we deduce that
$$
q_i^{\frac{1}{2}}a_w(1-q_i^{-1}x^{w\alpha_i^{\vee}})=a_{ws_i}\qquad\textrm{whenever $\ell(ws_i)=\ell(w)-1$}.
$$
Write $w_0=ws_{i_1}\cdots s_{i_{\ell}}$ with $\ell=\ell(w_0)-\ell(w)$. Then
\begin{align*}
a_w&=q_{i_1}^{\frac{1}{2}}a_{ws_{i_1}}(1-q_{i_1}^{-1}x^{-w\alpha_{i_1}^{\vee}})=\cdots=q_{w^{-1}w_0}^{1/2}a_{w_0}\prod_{\alpha}(1-q_{\alpha}^{-1}x^{-w\alpha^{\vee}}),
\end{align*}
where the product is over $\alpha\in \{\alpha_{i_1},s_{i_1}\alpha_{i_2},\ldots,s_{i_1}\cdots s_{i_{\ell-1}}\alpha_{i_{\ell}}\}=R(w^{-1}w_0)$.
\end{proof}

\begin{lemma}\label{lem:tinylemma2}
Let $\rho=\frac{1}{2}\sum_{\alpha\in R^+}\alpha^{\vee}$. If $w\in W_0$ then
$$
\tau_w\mathbf{1}_0=\bigg[(-1)^{\ell(w)}q_w^{\frac{1}{2}}x^{-\rho+w\rho}\prod_{\beta\in R(w^{-1})}(1-q_{\beta}^{-1}x^{-w\beta^{\vee}})\bigg]\mathbf{1}_0.
$$
\end{lemma}

\begin{proof} It follows from (\ref{eq:later}) that if $w=s_{i_1}\cdots s_{i_{\ell}}$ is reduced then
$$
\tau_w\mathbf{1}_0=\bigg(\prod_{\alpha}\left[-q_{\alpha}^{\frac{1}{2}}x^{-\alpha^{\vee}}(1-q_{\alpha}^{-1}x^{\alpha^{\vee}})\right]\bigg)\mathbf{1}_0,
$$
where the product is over $\alpha\in\{\alpha_{i_1},s_{i_1}\alpha_{i_2},\ldots,s_{i_1}\cdots s_{i_{\ell-1}}\alpha_{i_{\ell}}\}=R(w)$. Since $R(w)=-wR(w^{-1})$ it follows that
$$
\tau_w\mathbf{1}_0=\bigg[(-1)^{\ell(w)}q_w^{\frac{1}{2}}x^{\sum_{\beta\in R(w^{-1})}w\beta^{\vee}}\prod_{\beta\in R(w^{-1})}(1-q_{\beta}^{-1}x^{-w\beta^{\vee}})\bigg]\mathbf{1}_0,
$$
and the result follows since $w\rho-\rho=\sum_{\beta\in R(w^{-1})}w\beta^{\vee}$.
\end{proof}

\begin{thm}\label{thm:macdonald}
For all $\mu\in P$ we have the Macdonald formula
$$
\mathbf{1}_0x^{\mu}\mathbf{1}_0=P_{\mu}(x)\mathbf{1}_0\quad\textrm{where}\quad P_{\mu}(x)=\frac{q_{w_0}}{W_0(q)}\sum_{w\in W_0}w\bigg(x^{\mu}\prod_{\alpha\in R^+}\frac{1-q_{\alpha}^{-1}x^{-\alpha^{\vee}}}{1-x^{-\alpha^{\vee}}}\bigg).
$$
\end{thm}

\begin{proof}
By Theorem~\ref{thm:1_0} and Lemma~\ref{lem:tinylemma2} we have
\begin{align*}
d(x)\mathbf{1}_0x^{\mu}\mathbf{1}_0&=\frac{q_{w_0}}{W_0(q)}\sum_{w\in W_0}q_w^{-\frac{1}{2}}c_wx^{w\mu}\tau_w\mathbf{1}_0\\
&=\frac{q_{w_0}}{W_0(q)}x^{-\rho}\sum_{w\in W_0}(-1)^{\ell(w)}(wn(x))x^{w\mu+w\rho}\mathbf{1}_0.
\end{align*}
By Bourbaki \cite[VI, \S3, No.3, Proposition~2]{bourbaki} the polynomial
$$
p(x)=\sum_{w\in W_0}(-1)^{\ell(w)}(wn(x))x^{w\mu+w\rho}\quad\textrm{is divisible by}\quad x^{\rho}d(x),
$$
and since $w(x^{\rho}d(x))=(-1)^{\ell(w)}x^{\rho}d(x)$ we have 
$$
\frac{p(x)}{x^{\rho}d(x)}=\sum_{w\in W_0}x^{w\mu}\prod_{\alpha\in R^+}\frac{1-q_{\alpha}^{-1}x^{-w\alpha^{\vee}}}{1-x^{-w\alpha^{\vee}}},
$$
completing the proof.
\end{proof}

\begin{remark} The above computation can be used to prove the \textit{Satake isomorphism} 
$$\mathbf{1}_0\scH\mathbf{1}_0\cong Z(\scH)=\CC[P]^{W_0},
$$
because $\{\mathbf{1}_0x^{\la}\mathbf{1}_0\mid \la\in P^+\}$ is a basis for $\mathbf{1}_0\scH\mathbf{1}_0$ and $\{P_{\la}(x)\mid\la\in P^+\}$ is a basis for $\CC[P]^{W_0}$.
\end{remark}

\subsection{Some representation theory} 

The calculation of the centre of $\scH$ has important implications for the representation theory of $\scH$.

\begin{prop}\label{prop:fundamental} Let $(\pi,V)$ be an irreducible representation of $\scH$ over $\CC$.
\begin{enumerate}
\item There is an element $t\in \Hom(P,\CC^{\times})$ such that
$$
\pi(z)=h_t(z)I\qquad\textrm{for all $z\in\CC[P]^{W_0}$},
$$
where $h_t:\CC[P]\to\CC$ is the evaluation homomorphism given by
$$
h_t(x^{\la})=t^{\la}\qquad\textrm{for all $\la\in P$, where $t^{\la}:=t(\la)$}.
$$
\item We have $h_t(z)=h_{t'}(z)$ for all $z\in \CC[P]^{W_0}$ if and only if $t'\in W_0t$. 
\item $V$ is necessarily finite dimensional.
\end{enumerate}
\end{prop}

\begin{proof} 1. The algebra $\scH$ has countable dimension. Therefore by Dixmier's infinite dimensional generalisation of Schur's Lemma (see \cite[\S5.3, Lemma~9]{varadarajan}), if $(\pi,V)$ is an irreducible representation of $\scH$ then $Z(\scH)=\CC[P]^{W_0}$ acts on $V$ by scalars. Thus $\pi:\scH\to GL(V)$ induces an algebra homomorphism $h:\CC[P]^{W_0}\to\CC$ by $\pi(z)=h(z)I$ for all $z\in\CC[P]^{W_0}$. Since $\CC[P]$ is integral over $\CC[P]^{W_0}$ each algebra homomorphism $h:\CC[P]^{W_0}\to\CC$ is the restriction of some homomorphism $\CC[P]\to\CC$. Therefore $h=h_t$ for some $t\in\Hom(P,\CC^{\times})$.

2. Exercises 12 and 13 in \cite[Chapter~V]{atiyah} show that the homomorphisms $h_t$ and $h_{t'}$ agree on $\CC[P]^{W_0}$ if and only if $t'\in W_0t$.

3. Since $\CC[P]$ is integral over $\CC[P]^{W_0}$, and since $\{T_wx^{\la}\mid w\in W_0,\la\in P\}$ is a vector space basis of $\scH$, it follows that $\scH$ is finite dimensional as a $\CC[P]^{W_0}$-module, and hence $V$ is finite dimensional (by part 1.).
\end{proof}

\begin{remark} It can be shown that if $(\pi,V)$ is an irreducible representation of $\scH$ then $\dim(V)\leq |W_0|$ (see \cite{kato}).
\end{remark}

\begin{defn} The element $t\in \Hom(P,\CC^{\times})$ in Proposition~\ref{prop:fundamental} is the \textit{central character} of $(\pi,V)$. To be more precise, the central character is the orbit $W_0t$.
\end{defn}

Note that $\scH=\scH_0\otimes\CC[P]$ where $\scH_0$ is the $|W_0|$-dimensional subalgebra generated by $T_w$, $w\in W_0$. This allows us to write down finite dimensional representations of $\scH$ by \textit{inducing} representations of the commutative subalgebra $\CC[P]$ to $\scH$. For $t\in\Hom(P,\CC^{\times})$ let $\CC v_t$ be the one dimensional representation of $\CC[P]$ with action $x^{\la}\cdot v_t=t^{\la}v_t$.

\begin{defn}
Let $t\in\Hom(P,\CC^{\times})$. The \textit{principal series representation of $\scH$ with central character $t$} is $(\pi_t,V(t))$, where 
$$
V(t)=\mathrm{Ind}_{\CC[P]}^{\scH}(\CC v_t)=\scH\otimes_{\CC[P]} (\CC v_t).
$$
We have $h\cdot(h'\otimes v_t)=(hh'\otimes v_t)$ and $(x^{\la}\otimes v_t)=t^{\la}(1\otimes v_t)$. Therefore $V(t)$ has basis $\{(T_w\otimes v_t)\mid w\in W_0\}$, and hence has dimension $|W_0|$.
\end{defn}

The importance of these representations is given by Kato's Theorem:

\begin{thm}[see \cite{kato}] We have
\begin{enumerate}
\item $(\pi_t,V(t))$ is irreducible if and only if $t^{\alpha^{\vee}}\neq q^{\pm1}$ for each $\alpha\in R$.
\item If $(\pi,V)$ is an irreducible representation of $\scH$ with central character $t$ then $(\pi,V)$ is a composition factor of $(\pi_t,V(t))$. 
\end{enumerate}
\end{thm}

See Section~\ref{sect:8} for explicit computations of representations in type $\tilde{A}_2$. Another way of building representations is to induce representations from \textit{parabolic subalgebras}. Agin, see Section~\ref{sect:8} for some examples.

We now give a simple combinatorial formula for the matrix elements of the principal series representation. Recall the definitions of $\wt(p)\in P$ and $\theta(p)\in W_0$ from~(\ref{eq:finaleq}). It is convenient to generalise the definition of $\cP(\vec w)$ from (\ref{eq:cP}) to allow alcove walks that start at an alcove different from~$1$. Given $u\in W_0$ and $\vec w\in\tilde{W}$ a reduced expression, let
$$
\cP(\vec w,u)=\{\textrm{positively alcove walks of type $\vec w$ starting at the alcove $u$}\}.
$$

\begin{thm}\label{thm:combform} Let $w\in\tilde{W}$. Relative to the basis $\{(T_u^{-1}T_{w_0}\otimes v_t)\mid u\in W_0\}$, the matrix elements of the principal series representation $(\pi_t,V(t))$ are given by
$$
[\pi_t(T_{w^{-1}})]_{v,u}=\sum_{\{p\in\cP(\vec{w},u)\mid \theta(p)=v\}}\cQ(p)t^{-w_0(\wt(p))}
$$
where $\cQ(p)$ is as in Proposition~\ref{prop:pathformula}.
\end{thm}

\begin{proof}
If $u\in W_0$ and $w\in \tilde{W}$, then by Proposition~\ref{prop:pathformula} we have
$$
(T_{w^{-1}}T_u^{-1})^*=T_{u^{-1}}^{-1}T_w=\sum_{p\in\cP(\vec{w},u)}\cQ(p)x^{\wt(p)}T_{\theta(p)^{-1}}^{-1}
$$
(the involution $*$ is described in (\ref{eq:star})). Since $(x^{\la})^*=T_{w_0}x^{-w_0\la}T_{w_0}^{-1}$ we get
\begin{align*}
T_{w^{-1}}\cdot (T_{u}^{-1}T_{w_0}\otimes v_t)&=T_{w^{-1}}T_u^{-1}\cdot(T_{w_0}\otimes v_t)\\
&=\sum_{p\in\cP(\vec{w},u)}\cQ(p)t^{-w_0(\wt(p))}(T_{\theta(p)}^{-1}T_{w_0}\otimes v_t),
\end{align*}
and the result follows.
\end{proof}

The positivity of this formula has some very useful applications, for example see the proof of Lemma~\ref{lem:bound1}. 

\section{Harmonic analysis for the Hecke algebra}\label{sect:6}

In the previous section we recalled some of the well known structural theory of affine Hecke algebras. In the current section we describe the beginnings of the harmonic analysis on $\scH$, following the main line of argument in~\cite{opdamtrace}. The outline is as follows. Define a trace $\Tr:\scH\to\CC$ on $\scH$ by linearly extending $\Tr(T_w)=\delta_{w,1}$. For fixed $t\in\Hom(P,\CC^{\times})$ define a function $F_t:\scH\to\CC$ by 
\begin{align}\label{eq:f1}
F_t(h)=\sum_{\mu\in P}t^{-\mu}\Tr(x^{\mu}h)\qquad\textrm{whenever the series converges}.
\end{align}
We show that the series converges provided each $|t^{\alpha_i^{\vee}}|<r$ is sufficiently small. We will see that
\begin{align}\label{eq:f2}
F_t(h)=\frac{f_t(h)}{q_{w_0}c(t)c(t^{-1})},\quad\textrm{where}\quad c(t)=\frac{n(t)}{d(t)}=\prod_{\alpha\in R^+}\frac{1-q_{\alpha}^{-1}t^{-\alpha^{\vee}}}{1-t^{-\alpha^{\vee}}},
\end{align}
where $d(t)$ and $n(t)$ are as in (\ref{eq:DN}), and where $d(t)f_t(h)$ is a linear combination of terms $\{t^{\la}\mid\la\in P\}$. Hence $f_t(h)$ has a meromorphic continuation (as a function of~$t$). Furthermore $f_t$ is related to the character of the principal series representation $(\pi_t,V(t))$ by $\tilde{f}_t=\chi_t$, where $\tilde{f}_t$ is the \textit{symmetrisation} of $f_t$. It follows from (\ref{eq:f1}) and (\ref{eq:f2}) that if $dt$ is normalised Haar measure on the product $\TT^n$ of $n$ circle groups $\TT$ then
$$
\Tr(h)=\int_{(r\TT)^n}\frac{f_t(h)}{q_{w_0}c(t)c(t^{-1})}\,dt.
$$
A more general version of this formula is the main result of \cite{opdamtrace}, and it is at the heart of the harmonic analysis and Plancherel measure for~$\scH$.

\subsection{The $C^*$-algebra}

Define an involution $*$ on $\scH$ and a function $\Tr:\scH\to\CC$ by
\begin{align}\label{eq:star}
\bigg(\sum_{w\in \tilde{W}}c_wT_w\bigg)^*=\sum_{w\in \tilde{W}}\overline{c_w}T_{w^{-1}}\quad\textrm{and}\quad\Tr\bigg(\sum_{w\in \tilde{W}}c_wT_w\bigg)=c_1.
\end{align}
An induction on $\ell(v)$ using the defining relations in the algebra $\scH$ shows that that $\Tr(T_u^*T_v)=\delta_{u,v}$, and so
\begin{align}\label{eq:tracecommute}
\Tr(h_1h_2)=\Tr(h_2h_1)\qquad\textrm{for all $h_1,h_2\in\scH$}.
\end{align}
It follows that
$$
(h_1,h_2):=\Tr(h_1^*h_2)
$$
defines a Hermitian inner product on $\scH$. Let $\norm h\norm_2=\sqrt{(h,h)}$. The algebra $\scH$ acts on itself, and the corresponding operator norm is
$$
\norm h\norm=\sup\{\norm hx\norm_2\,:\,x\in\scH,\norm x\norm_2\leq 1\}.
$$
Let $\overline{\scH}$ denote the completion of $\scH$ with respect to this norm. It is a non-commutative $C^*$-algebra.

Recall from Remark~\ref{rem:isomorphism} that if there is an underlying building then there is an isomorphism $\psi:\scH_W\to\scA$, where $\scH_W$ is the subalgebra of $\scH$ generated by $\{T_w\mid w\in W\}$. The isomorphism is given by $\psi(T_w)=q_w^{1/2}A_w$ for all $w\in W$. It is not immediately clear that the operator norms on $\scA$ and $\scH_W$ (written as $\|\cdot\|$ and $\norm\cdot\norm$ respectively) are compatible with $\psi$, and so we pause to prove the following:

\begin{prop} If there is an underlying building then $\norm h\norm=\|\psi(h)\|$ for all $h\in\scH_W$. Therefore $\overline{\scH}_W\cong\overline{\scA}$.
\end{prop}

\begin{proof}

Let $o\in\cC$ be a fixed chamber of the underlying building $(\cC,\delta)$. Let $\ell^2_o(\cC)$ be the subspace of $\ell^2(\cC)$ consisting of functions which are constant on each set $\cC_w(o)$. Since $A_w\delta_o=q_w^{-1}1_{\cC_{w^{-1}}(o)}$ the injective map $\omega:\scA\to\ell^2_o(\cC)$, $A\mapsto A\delta_o$, embeds $\scA$ into $\ell^2_o(\cC)$ as a dense subspace (the subspace of finitely supported functions). Therefore $\eta:=\omega\circ \psi:\scH_W\to\ell_o^2(\cC)$ embeds $\scH_W$ as a dense subspace of $\ell_o^2(\cC)$, and a straight forward computation shows that
$\norm h\norm_2=\|\eta(h)\|_2$ for all $h\in\scH_W$. Therefore by the density of $\eta(\scH_W)$ in $\ell_o^2(\cC)$ it follows that
$$
\norm h\norm=\|\psi(h)\|_o\qquad\textrm{for all $h\in\scH_W$},
$$
where $\|\cdot\|_o$ is the $\ell^2$-operator norm on $\scB(\ell_o^2(\cC))$ (note that each $A\in\scA$ maps $\ell_o^2(\cC)$ into itself). It remains to show that $\|A\|_o=\|A\|$ for all $A\in\scA$. To see this, note that the homomorphism $\Phi:\overline{\scA}\to\scB(\ell_o^2(\cC))$, $A\mapsto A|_{\ell_o^2(\cC)}$ is injective, for if $\Phi(A)=0$ then $A\delta_o=0$, and so $A=0$ by Lemma~\ref{lem:small}. But by \cite[Theorem~I.5.5]{davidson} an injective homomorphism between $C^*$-algebras is necessarily an isometry, and so $\|A\|=\|A\|_o$ for all $A\in{\scA}$. 
\end{proof}

We return to the study of the trace functional $\Tr:\scH\to\CC$.

\subsection{A formula for the trace on $\scH\mathbf{1}_0$}

It is easy to derive a formula for the trace on $\scH\mathbf{1}_0=\CC[P]\mathbf{1}_0$ using the harmonic analysis on $Z(\scH)$. The key idea is that $\Tr(x^{\mu}\mathbf{1}_0)=\Tr(x^{\mu}\mathbf{1}_0^2)=\Tr(\mathbf{1}_0x^{\mu}\mathbf{1}_0)$, and then $\mathbf{1}_0x^{\mu}\mathbf{1}_0=P_{\mu}(x)\mathbf{1}_0$, where $P_{\mu}(x)\in Z(\scH)$ is the Macdonald spherical function. First a formula for the trace on $Z(\scH)\mathbf{1}_0$.

\begin{thm}\label{thm:plancherelcenter} 
Let $p(x)\in\CC[P]^{W_0}$. Then
$$ 
\Tr(p(x)\mathbf{1}_0)=\frac{W_0(q)}{|W_0|q_{w_0}}\int_{\TT^n}\frac{p(t)}{c(t)c(t^{-1})}\,dt.
$$
\end{thm}

\begin{proof} Since $\{P_{\la}(x)\mid \la\in P^+\}$ is a basis for $\CC[P]^{W_0}$ it suffices to check the formula when $p(x)=P_{\la}(x)$ for some $\la\in P^+$. 
It is not hard to see that for $\lambda\in P^+$ we have
$$
\mathbf{1}_0x^{\la}\mathbf{1}_0=q_{t_{\la}}^{-1/2}\frac{q_{w_0}W_{0\la}(q^{-1})}{W_0(q)^2}\sum_{w\in W_0t_{\la}W_0}q_w^{\frac{1}{2}}T_w,
$$
where $W_{0\la}=\{w\in W_0\mid w\la=\la\}$ (see \cite[Lemma~2.7]{ram1} for example). Therefore using Theorem~\ref{thm:macdonald} we compute
$$
\Tr(P_{\la}(x)\mathbf{1}_0)=\Tr(\mathbf{1}_0x^{\la}\mathbf{1}_0)=\delta_{\la,0}
$$
where $\delta_{\la,\mu}$ is the Kroneker delta. On the other hand since $c(t)c(t^{-1})$ is $W_0$-invariant we have
\begin{align*}
\frac{W_0(q)}{|W_0|q_{w_0}}\int_{\TT^n}\frac{P_{\la}(t)}{c(t)c(t^{-1})}dt&=\frac{1}{|W_0|}\sum_{w\in W_0}\int_{\TT^n}\frac{t^{w\mu}c(wt)}{c(wt)c(wt^{-1})}dt=\int_{\TT^n}\frac{t^{\mu}}{c(t^{-1})}dt.
\end{align*}
If $\la\in P^+$ then $\la\in\QQ_{\geq0}\alpha_1^{\vee}+\cdots+\QQ_{\geq0}\alpha_n^{\vee}$ (see \cite[VI, \S1, No.10]{bourbaki}). Therefore if $\la\neq0$ is dominant then the integral is zero by regarding it as an iterated contour integral of a function that is analytic inside the contours in each variable $t^{\omega_1},\ldots,t^{\omega_n}$. The result follows.
\end{proof}

\begin{cor}\label{cor:tf} If $\mu\in P$ then
$$
\Tr(x^{\mu}\mathbf{1}_0)=\int_{\TT^n}\frac{t^{\mu}}{c(t^{-1})}\,dt.
$$
If $\mu\notin-Q^+$ then $\Tr(x^{\mu}\mathbf{1}_0)=0$.
\end{cor}

\begin{proof}
If $\mu\in P$ then by (\ref{eq:1}), (\ref{eq:tracecommute}), and Theorem~\ref{thm:macdonald} we have
\begin{align*}
\Tr(x^{\mu}\mathbf{1}_0)=\Tr(x^{\mu}\mathbf{1}_0^2)=\Tr(\mathbf{1}_0x^{\mu}\mathbf{1}_0)=\Tr(P_{\mu}(x)\mathbf{1}_0),
\end{align*}
where $P_{\mu}(x)\in\CC[P]^{W_0}$. The result now follows from Theorem~\ref{thm:plancherelcenter} and its proof.
\end{proof}

\subsection{Opdam's trace generating function formula}

Let $t\in\Hom(P,\CC^{\times})$. Define a function $F_t:\scH\to\CC$ by (\ref{eq:f1}). Let us deal immediately with the issue of convergence of this series. Recall that if $v\in\tilde{W}$ then $v\in t_{\mu}W_0$ for a unique $\mu\in P$, and we write $\wt(v)=\mu$.

\begin{lemma} Let $v\in\tilde{W}$. Then: \label{lem:technical}
\begin{enumerate}
\item $|\Tr(x_v)|\leq 2^{\ell(v)}q_{v}^{1/2}$.
\item If $\Tr(x_v)\neq0$ then $\mathrm{wt}(v)\in -Q^+$.  
\end{enumerate}
\end{lemma}

\begin{proof}
1. We use induction on $\ell(v)$ to prove that
$$
x_v=\sum_{w\in W}c_w^{v}T_w\qquad\textrm{with $|c_w^v|\leq 2^{\ell(v)}q_v^{1/2}$ for all $v,w$.}
$$
The result follows since $\Tr(x_v)=c_1^v$. The case $\ell(v)=0$ is trivial, since $x_{\gamma}=T_{\gamma}$ for all $\gamma\in\Gamma$. Suppose that $\ell(vs_i)=\ell(v)+1$. Then
\begin{align*}
x_{vs_i}=x_{v}T_{i}^{\epsilon}=\sum_{w\in W}c_w^vT_wT_{i}^{\epsilon}\qquad\textrm{where $\epsilon=+1$ or $\epsilon=-1$}.
\end{align*}
If $\epsilon=+1$ then the relations in $\scH$ give
$$
c_{w}^{vs_i}=\begin{cases}
c_{ws_i}^v&\textrm{if $\ell(ws_i)>\ell(w)$}\\
c_{ws_i}^v+(q_i^{\frac{1}{2}}-q_i^{-\frac{1}{2}})c_w^v&\textrm{if $\ell(ws_i)<\ell(w)$},
\end{cases}
$$
and if $\epsilon=-1$ the relations in $\scH$ give
$$
c_w^{vs_i}=\begin{cases}
c_{ws_i}^v-(q_i^{\frac{1}{2}}-q_i^{-\frac{1}{2}})c_w^v&\textrm{if $\ell(ws_i)>\ell(w)$}\\
c_{ws_i}^v&\textrm{if $\ell(ws_i)<\ell(w)$}.
\end{cases}
$$
Therefore in all cases the induction hypothesis implies that
$$
|c_w^{vs_i}|\leq|c_{ws_i}^v|+(q_i^{\frac{1}{2}}-q_i^{-\frac{1}{2}})|c_w^v|\leq (1+q_i^{\frac{1}{2}})2^{\ell(v)}q_v^{1/2}\leq 2^{\ell(vs_i)}q_{vs_i}^{1/2}.
$$

2. Recall the path theoretic formula from Proposition~\ref{prop:pathformula}. Using (\ref{eq:endbruhat}) and (\ref{eq:enddominance}) this formula shows that for all $v\in\tilde{W}$ we have
 $$
 T_v=x_v+\sum_{\{u\mid u< v,\wt(u)\succeq\wt(v)\}}a_u^vx_u.
 $$
Inverting this change of basis formula gives
 $$
 x_v=T_v+\sum_{\{u\mid u< v,\wt(u)\succeq\wt(v)\}}c_u^vT_u.
 $$
 But $\Tr(x_v)=c_1^v$, and hence $\Tr(x_v)\neq 0$ implies that $\wt(v)\preceq\wt(1)=0$. Therefore by the definition of the dominance order we have $\wt(v)\in-Q^+$.
\end{proof}

\begin{cor}\label{cor:converge}
There exists $r>0$ such that for all $h\in\scH$ the series $F_t(h)$ converges uniformly if each $|t^{\alpha_i^{\vee}}|<r$.
\end{cor}

\begin{proof} This is immediate from Lemma~\ref{lem:technical}.
\end{proof}

Let $\TT$ be the circle group, and let $dt=dt_1\cdots dt_n$ be the normalised Haar measure on $\TT^n$. Let $r>0$ be as in Corollary~\ref{cor:converge}, and write $\TT_r=r\TT$. Then
$$
\Tr(h)=\int_{\TT_r^n}F_t(h)\,dt\qquad\textrm{for all $h\in\scH$}.
$$
This is the starting point for the harmonic analysis on $\scH$. Our first task is to compute $F_t(h)$. The following very nice properties are useful.

\begin{prop}\label{prop:F_t} Let $t\in\Hom(P,\CC^{\times})$ with $|t^{\alpha_i^{\vee}}|<r$ for each $i=1,\ldots,n$. The function $F_t:\scH\to\CC$ satisfies:
\begin{enumerate}
\item $F_t$ is linear.
\item $F_t(x^{\la}hx^{\mu})=t^{\la+\mu}F_t(h)$ for all $\la,\mu\in P$ and all $h\in\scH$.
\item $F_t(\tau_w)=\delta_{w,1}F_t(1)$ for all $w\in W_0$.
\end{enumerate}
\end{prop}

\begin{proof}
The first statement is obvious. For the second statement, using the fact that $\Tr(ab)=\Tr(ba)$ and making a change of variable in the summation gives
\begin{align*}
F_t(x^{\la}hx^{\mu})&=\sum_{\nu\in P}t^{-\nu}\Tr(x^{\la+\nu}hx^{\mu})=\sum_{\nu\in P}t^{-\nu}\Tr(x^{\la+\mu+\nu}h)=t^{\la+\mu}F_t(h).
\end{align*}
Finally, since $\tau_wx^{\la}=x^{w\la}\tau_w$ for all $\la\in P$ and $w\in W_0$, we have
$$
t^{\la}F_t(\tau_w)=F_t(\tau_wx^{\la})=F_t(x^{w\la}\tau_w)=t^{w\la}F_t(\tau_w),
$$
and so $(t^{\la}-t^{w\la})F_t(\tau_w)=0$. The condition on $t\in\Hom(P,\CC^{\times})$ implies that if $w\neq 1$ then $x^{\la}-x^{w\la}\neq 0$ for all $\la\in P$, and the result follows.
\end{proof}

\begin{prop}\label{prop:6.7} Let $t\in\Hom(P,\CC^{\times})$ with $|t^{\alpha_i^{\vee}}|<r$ for each $i=1,\ldots,n$. Then
$$
F_t(h)=f_t(h)F_t(1)\qquad\textrm{for all $h\in\scH$},
$$
where $d(t)f_t(h)$ is a polynomial in $\{t^{\la}\mid \la\in P\}$ with complex coefficients.
\end{prop}

\begin{proof}
If $w\in W_0$ then $d(x)T_w$ can be written as a linear combination of the elements $\{\tau_v\mid v\in W_0\}$ with complex coefficients. Therefore if $h\in\scH$ then $d(x)h$ is in the $\CC$-span of $\{x^{\la}\tau_w\mid w\in W_0,\la\in P\}$. Writing
$$
d(x)h=\sum_{w\in W_0}p_w(x)\tau_w\qquad\textrm{with $p_w(x)\in\CC[P]$}
$$
and using Proposition~\ref{prop:F_t} gives
\begin{align}
d(t)F_t(h)&=F_t(d(x)h)=\sum_{w\in W_0}p_w(t)F_t(\tau_w)=p_1(t)F_t(1).\qedhere
\end{align}
\end{proof}

Proposition~\ref{prop:6.7} shows that for each $h\in\scH$, the function $t\mapsto f_t(h)$ has a meromorphic continuation to $t\in\Hom(P,\CC^{\times})$ with possible poles at the points where $t^{\alpha^{\vee}}=1$ for some $\alpha\in R$. Proposition~\ref{prop:F_t} immediately implies the following.

\begin{cor}\label{cor:littlef}
The function $f_t:\scH\to\CC$ satisfies:
\begin{enumerate}
\item $f_t$ is linear.
\item $f_t(x^{\la}hx^{\mu})=t^{\la+\mu}f_t(h)$ for all $\la,\mu\in P$ and all $h\in\scH$.
\item $f_t(x^{\la}\tau_w)=t^{\la}\delta_{w,1}$ for all $\la\in P$ and $w\in W_0$.
\end{enumerate}
\end{cor}

Let $\tilde{f}_t:\scH\to\CC$ be the \textit{symmetrisation} of $f_t$. That is,
$$
\tilde{f}_t(h)=\sum_{w\in W_0}f_{wt}(h)\qquad\textrm{for all $h\in \scH$}.
$$
The following theorem gives an important connection between $\tilde{f}_t$ and the character $\chi_t$ of the principal series representation. First a quick lemma.

\begin{lemma}\label{lem:littlelemma} If $t^{\alpha^{\vee}}\neq 1$ for all $\alpha\in R$ then $V(t)$ has basis $\{\tau_w\otimes v_t\mid w\in W_0\}$ and 
$$
\chi_t(x^{\la}\tau_w)=\delta_{w,1}\sum_{w'\in W_0}t^{w't}.
$$
\end{lemma}

\begin{proof} For each $w\in W_0$ induction shows that $d(x)T_w$ can be written as a linear combination of elements $\{x^{\la}\tau_{w'}\mid w'\in W_0,\la\in P\}$, and the first claim follows. 

For all $\la\in P$ and $w,u\in W_0$ we have
$$
(x^{\la}\tau_w)\cdot(\tau_u\otimes v_t)=(\tau_w\tau_ux^{u^{-1}w^{-1}\la}\otimes v_t)=t^{u^{-1}w^{-1}\la}(\tau_w\tau_v\otimes v_t).
$$
It follows from (\ref{eq:iii}) that for all $w,u\in W_0$ we have $\tau_w\tau_u\in\tau_{wu}\CC[P]$. Therefore if $w\neq 1$ then diagonal entries of $\pi_t(x^{\la}\tau_w)$ relative to the basis $\{\tau_w\otimes v_t\mid w\in W_0\}$ are all zero, and so $\chi_t(x^{\la}\tau_w)=0$. If $w=1$ then the diagonal of $\pi_t(x^{\la})$ consists of the terms $t^{w'\la}$ with $w'\in W_0$, and the result follows.
\end{proof} 

\begin{thm}\label{thm:characters}  If $t^{\alpha^{\vee}}\neq 1$ for all $\alpha\in R$ then $\tilde{f}_t(h)=\chi_t(h)$ for all $h\in\scH$.
\end{thm}

\begin{proof} 
Since $\CC[P]^{W_0}=Z(\scH)$ it is clear that $\pi_t(p(x))=p(t)I$ for all $p\in\CC[P]^{W_0}$. Therefore for all $h\in\scH$ we have
$$
\chi_t(p(x)h)=\tr(\pi_t(p(x))\pi_t(h))=\tr(p(t)\pi_t(h))=p(t)\chi_t(h),
$$
and Corollary~\ref{cor:littlef} gives
$$
\tilde{f}_t(p(x)h)=p(t)\tilde{f}_t(h).
$$

Pick $p(x)=d(x)d(x^{-1})$ (this is symmetric). Since $d(x)T_w$ is in the $\CC[P]$-span of $\{\tau_v\mid v\in W_0\}$ for all $w\in W_0$ we see that $p(x)h$ can be written as
$$
p(x)h=\sum_{w\in W_0}p_w(x)\tau_w\qquad\textrm{with $p_w(x)\in\CC[P]$}.
$$
By Corollary~\ref{cor:littlef}, Lemma~\ref{lem:littlelemma}, and the above observations we see that
$$
p(t)\chi_t(h)=\chi_t(p(x)h)=\sum_{w'\in W_0}p_1(wt)=\tilde{f}_t(p(x)h)=p(t)\tilde{f}_t(h).
$$
So if $t^{\alpha^{\vee}}\neq 1$ for all $\alpha\in R$ we have $\tilde{f}_t(h)=\chi_t(h)$ (since $p(t)=d(t)d(t^{-1})\neq 0$).
\end{proof}

Since $F_t(h)=f_t(h)F_t(1)$ the final piece in the puzzle is to compute $F_t(1)$.

\begin{thm}\label{thm:traceformula} Let $t\in\Hom(P,\CC^{\times})$ be such that $F_t(1)$ converges. Then
$$
F_t(1)=\frac{q_{w_0}^{-1}}{c(t)c(t^{-1})},\qquad\textrm{where}\qquad c(t)=\frac{n(t)}{d(t)}=\prod_{\alpha\in R^+}\frac{1-q_{\alpha}^{-1}t^{-\alpha^{\vee}}}{1-t^{-\alpha^{\vee}}}.
$$
Thus for all $h\in\scH$ the series $F_t(h)$ converges if $|t^{\alpha^{\vee}}|<q_{\alpha}^{-1}$ for each $\alpha\in R^+$.
\end{thm}

\begin{proof} Let $\mu\in P$. By Corollary~\ref{cor:tf} and the definition of $F_t$ we have
\begin{align*}
F_t(\mathbf{1}_0)=\sum_{\mu\in -Q^+}t^{-\mu}\Tr(x^{\mu}\mathbf{1}_0)=\sum_{\mu\in -Q^+}\int_{\TT^n}\frac{t^{-\mu}u^{\mu}}{c(u^{-1})}du=\int_{\TT^n}\prod_{i=1}^n\frac{u^{\alpha_i^{\vee}}}{u^{\alpha_i^{\vee}}-t^{\alpha_i^{\vee}}}\frac{du}{c(u^{-1})}.
\end{align*}
Considering this integral as a iterated contour integral, and computing the simple residues at $u^{\alpha_i^{\vee}}=t^{\alpha_i^{\vee}}$ gives
\begin{align}\label{eq:littlecomp}
F_t(\mathbf{1}_0)=\frac{1}{c(t^{-1})}.
\end{align}
By Theorem~\ref{thm:1_0} we have
$$
d(x)\mathbf{1}_0=q_{w_0}n(x)+\sum_{w\in W_0, w\neq 1}a_w(x)\tau_w\quad\textrm{for some polynomials $a_w(x)\in\CC[P]$}
$$
and so by Proposition~\ref{prop:F_t} we have
\begin{align*}
d(t)F_t(\mathbf{1}_0)=F_t(d(x)\mathbf{1}_0)&=q_{w_0}F_t(n(x))=q_{w_0}n(t)F_t(1)
\end{align*}
and the result follows from (\ref{eq:littlecomp}).
\end{proof}

\section{The Plancherel Theorem for $\tilde{A}_2$}\label{sect:7}

Let us prove the Plancherel Theorem in the $\tilde{A}_2$ case. In fact we will prove the Plancherel formula for the algebra $\scH_W$ (this is slightly more convenient; similar computations work for the extended affine Hecke algebra). This is the affine Hecke algebra with generators $T_w$, $w\in W$, and relations
$$
T_wT_{s_i}=\begin{cases}
T_{ws_i}&\textrm{if $\ell(ws_i)=\ell(w)+1$}\\
T_{ws_i}+(q^{\frac{1}{2}}-q^{-\frac{1}{2}})T_w&\textrm{if $\ell(ws_i)=\ell(w)-1$},
\end{cases}
$$
where $W$ is the Coxeter group from Figure~\ref{fig:A_2}. Alternatively this algebra has basis $\{x^{\mu}T_w\mid w\in W_0,\mu\in Q\}$, where $W_0$ is the parabolic subgroup of $W$ generated by $s_1$ and $s_2$, and where $Q=\ZZ\alpha_1^{\vee}+\ZZ\alpha_2^{\vee}$. The relations that hold amongst the elements of this basis are given by Theorem~\ref{thm:pres}.

Summarising the results of the last section, for all $h\in\scH_W$ we have
\begin{align}\label{eq:startingpoint}
\Tr(h)=\int_{\TT_r^2}F_t(h)\,dt=\int_{\TT_r^2}f_t(h)F_t(1)\,dt=\int_{\TT_r^2}\frac{f_t(h)}{q_{w_0}c(t)c(t^{-1})}\,dt
\end{align}
where $r>0$ is sufficiently small, and where $f_t(h)=\phi_h(t)/d(t)$ for some linear combination $\phi_h(t)$ of the terms $t^{\la}$, $\la\in  Q$, with $\tilde{f}_t(h)=\chi_t(h)$.

Therefore we are lead to consider integrals of the form
$$
I_{f}=\int_{\TT_r^2}\frac{f(t)}{c(t)c(t^{-1})}\,dt\qquad\textrm{where $f(t)=\phi(t)/d(t)$ with $\phi$ analytic on $(\CC^{\times})^2$.}
$$

\begin{lemma}\label{lem:plancherel} Let $I_{f}$ be as above, and let $c(t)$ and $c_1(u)$ be as in Theorem~\ref{thm:plancherel}. Then
\begin{align*}
I_f=\int_{\TT^2}\frac{f(t)}{|c(t)|^2}dt+\frac{q(q-1)^2}{q^2-1}\int_{\TT}\frac{g_f(u)}{|c_1(u)|^2}du+\frac{q^3(q-1)^3}{q^3-1}f(q^{-1},q^{-1}),
\end{align*}
where $g_f(u)=f(q^{\frac{1}{2}}u,q^{-1})+f(q^{-\frac{1}{2}}u^{-1},q^{-\frac{1}{2}}u)+f(q^{-1},q^{\frac{1}{2}}u)$.
\end{lemma}

\begin{proof} This is just some residue calculus. Write
$$
I_{f}=\int_{\TT_r}I_{\phi}(t_2)\,dt_2\qquad\textrm{where}\qquad I_{\phi}(t_2)=\int_{\TT_r}\frac{\phi(t)d(t^{-1})}{n(t)n(t^{-1})}\,dt_1.
$$
Fix $t_2\in\TT_r$. Consider the integral $I_{\phi}(t_2)$ as a contour integral ($dt_1=\frac{1}{2\pi i}\frac{dz}{z}$) along $C_r$ (the circular contour with radius $r$ and centre $0$ traversed once counterclockwise). We will shift this contour to $C_1$. In doing so we will pick up residue contributions from the poles of the integrand lying between $C_r$ and $C_1$. Since
$n(t)=(1-q^{-1}t_1^{-1})(1-q^{-1}t_2^{-1})(1-q^{-1}t_1^{-1}t_2^{-1})$
we see that the only pole between $C_r$ and $C_1$ (for fixed $t_2\in\TT_r$) comes from the term $n(z,t_2)^{-1}$ at $z=q^{-1}$. Therefore a residue computation gives
\begin{align*}
I_{f}&=\int_{\TT}J_{\phi}(t_1)\,dt_1+\frac{q(q-1)}{q^2-1}K_{\phi},
\end{align*}
where
$$
J_{\phi}(t_1)=\int_{\TT_r}\frac{\phi(t)d(t^{-1})}{n(t)n(t^{-1})}\,dt_2\qquad\textrm{and}\qquad K_{\phi}=\int_{\TT_r}\frac{t_2\phi(q^{-1},t_2)}{(1-q^{-1}t_2^{-1})(1-q^{-2}t_2)}\,dt_2.
$$
For fixed $t_1\in\TT$, the poles of the integrand of $J_{\phi}(t_1)$ between the contours $C_r$ and $C_1$ are at $z=q^{-1}$ and $z=q^{-1}t_1^{-1}$. Residue calculus shows that $J_{\phi}(t_1)$ equals
$$
\int_{\TT}\frac{\phi(t)d(t^{-1})}{n(t)n(t^{-1})}dt_2+\frac{q(q-1)}{q^2-1}\bigg(\frac{t_1\phi(t_1,q^{-1})}{(1-q^{-1}t_1^{-1})(1-q^{-2}t_1)}-\frac{\phi(t_1,q^{-1}t_1^{-1})}{(1-q^{-1}t_1)(1-q^{-2}t_1^{-1})}\bigg).
$$
(This computation assumes $t_1\neq 1$. At $t_1=1$ the pole has order~$2$, but this is a set of measure zero.) Residue computations give
$$
K_{\phi}=\int_{\TT}\frac{t_2\phi(q^{-1},t_2)}{(1-q^{-1}t_2^{-1})(1-q^{-2}t_2)}\,dt_2-\frac{q^2}{q^3-1}\phi(q^{-1},q^{-1}).
$$
Putting these computations together and using the formula $\phi(t)=d(t)f(t)$ it follows that $I_f$ equals
\begin{align*}
\int_{\TT^2}\frac{f(t)}{c(t)c(t^{-1})}\,dt+\frac{q(q-1)^2}{q^2-1}\int_{\TT}\frac{(1-t)(1-qt^{-1})g_f'(t)}{(1-q^{-1}t^{-1})(1-q^{-2}t)}dt+\frac{q^3(q-1)^3}{q^3-1}f(q^{-1},q^{-1}).
\end{align*}
where $g'_f(t)=f(t,q^{-1})+f(t^{-1},q^{-1}t)+f(q^{-1},t)$. After a change of variable $t=q^{\frac{1}{2}}u$ in the second integral it becomes an integral over $q^{-\frac{1}{2}}\TT$, and the integrand has no poles between this contour and $\TT$. Therefore we can expand the contour to $\TT$ for free. The result follows.
\end{proof}

We have already encountered the representations $\pi_u^{(1)}$ and $\pi^{(2)}$. See the next section for the details.

\begin{lemma}\label{lem:characters2} We have
$$
f_{(q^{\frac{1}{2}}u,q^{-1})}(h)+f_{(q^{-\frac{1}{2}}u^{-1},q^{-\frac{1}{2}}u)}(h)+f_{(q^{-1},q^{\frac{1}{2}}u)}(h)=\chi_u^{(1)}(h)\quad\textrm{for all $h\in\scH_W$},
$$
and $f_{(q^{-1},q^{-1})}(h)=\chi^{(2)}(h)$ for all $h\in\scH_W$.
\end{lemma}

\begin{proof} The first statement is similar to Theorem~\ref{thm:characters}. Here is an outline (see the next section). If $u\neq q^{-\frac{1}{2}},q^{\frac{1}{2}}$ then the representation space $V(u)=\scH_W\otimes_{\scH_1}(\CC v_u)$ has basis $\{1\otimes v_u,\tau_2\otimes v_u,\tau_{s_1s_2}\otimes v_u\}$. One computes $\tau_1\otimes v_u=0$, and it easily follows that the diagonal entries of the matrices $\pi_u^{(1)}(\tau_w)$ with $w\neq 1$ are all zero. Therefore $\chi_u^{(1)}(\tau_w)=0$ for all $w\neq 1$. The result easily follows from Corollary~\ref{cor:littlef}.

To see that $f_{(q^{-1},q^{-1})}(h)=\chi^{(2)}(h)$ for all $h\in\scH_W$, note that $x^{\alpha_1^{\vee}}=T_2^{-1}T_0T_2T_1$ and $x^{\alpha_2^{\vee}}=T_1^{-1}T_0T_1T_2$, and so $\chi^{(2)}(x^{\alpha_1^{\vee}})=\chi^{(2)}(x^{\alpha_2^{\vee}})=q^{-1}$. On the other hand, $f_{(q^{-1},q^{-1})}(x^{k\alpha_1^{\vee}+\ell\alpha_2^{\vee}}\tau_w)=\delta_{w,1}q^{-k-\ell}$, and the result follows.
\end{proof}

We can now prove the Plancherel Theorem for $\tilde{A}_2$:

\begin{proof}[Proof of Theorem~\ref{thm:plancherel}]
By (\ref{eq:startingpoint}) and Lemmas~\ref{lem:plancherel} and \ref{lem:characters2} we have
$$
\Tr(h)=\frac{1}{q^3}\int_{\TT^2}\frac{f_t(h)}{|c(t)|^2}\,dt+\frac{(q-1)^2}{q^2(q^2-1)}\int_{\TT}\frac{\chi_u^{(1)}(h)}{|c_1(u)|^2}\,du+\frac{(q-1)^3}{q^3-1}\chi^{(2)}(h).
$$
The result follows from Theorem~\ref{thm:characters} (by symmetrising the first integral).
\end{proof}

\section{Some explicit representations}\label{sect:8}

Let us construct the representations $\pi_t$, $\pi_u^{(1)}$, and $\pi^{(2)}$ of $\scH_W$ that appear in the Plancherel Theorem. A good reference for the representation theory of rank~2 affine Hecke algebras (that is, those of type $\tilde{A}_1\times\tilde{A}_1$, $\tilde{A}_2$, $\tilde{B}_2$, and $\tilde{G}_2$) is \cite{ramrank2}.

\medskip

\noindent\textbf{Principal series representations.} The principal series representation with central character $t\in\Hom(P,\CC^{\times})$ is $(\pi_t,V(t))$, where $V(t)=\scH_W\otimes_{\CC[P]} \CC v_t$. Here $\CC v_t$ is the one dimensional representation of $\CC[P]$ given by $x^{\la}\cdot v_t=t^{\la}v_t$. The representation space $V(t)$ has basis $\{T_w\otimes v_t\mid w\in W_0\}$. 

The following example illustrates how to compute relative to this basis. Write $\fq=q^{\frac{1}{2}}-q^{-\frac{1}{2}}$. The Bernstein relation gives
$
x^{\alpha_1^{\vee}}T_{1}=T_1x^{-\alpha_1^{\vee}}+\fq(1+x^{\alpha_1^{\vee}}),
$
and so $x^{\alpha_1^{\vee}}$ acts on the basis element $T_1\otimes v_t$ by
\begin{align*}
x^{\alpha_1^{\vee}}\cdot(T_1\otimes v_t)&=\big(T_1x^{-\alpha_1^{\vee}}+\fq(1+x^{\alpha_1^{\vee}})\big)\otimes v_t=t^{-\alpha_1^{\vee}}(T_1\otimes v_t)+\fq(1+t^{\alpha_1^{\vee}})(1\otimes v_t).
\end{align*}
Computing the matrices $\pi_t(T_1)$ and $\pi_t(T_2)$, is straight forward, and we recover the matrices from Section~\ref{sect:statement} (remember that $A_w\leftrightarrow q_w^{-1/2}T_w$). One can compute $\pi_t(T_0)$ using $x^{\varphi^{\vee}}=T_0T_1T_2T_1$ where $\varphi=\alpha_1+\alpha_2$, but it is quicker to use
$$
T_0\cdot(T_w\otimes v_t)=\begin{cases}
t^{-w^{-1}\varphi^{\vee}}(T_{s_{\varphi}w}\otimes v_t)&\textrm{if $\varphi\in R(w)$}\\
t^{-w^{-1}\varphi^{\vee}}(T_{s_{\varphi}w}\otimes v_t)+\fq(T_w\otimes v_t)&\textrm{if $\varphi\notin R(w)$}
\end{cases}
$$
which follows from \cite[(3.3.6)]{macblue}. Therefore the matrix $\pi_t(T_0)$ is given by
\begin{align*}
T_0\cdot(1\otimes v_t)&=\fq(1\otimes v_t)+t^{-\varphi^{\vee}}(T_{s_1s_2s_1}\otimes v_t)\\
T_0\cdot(T_1\otimes v_t)&=\fq(T_1\otimes v_t)+t^{-\alpha_2^{\vee}}(T_{s_1s_2}\otimes v_t)\\
T_0\cdot(T_2\otimes v_t)&=\fq(T_2\otimes v_t)+t^{-\alpha_1^{\vee}}(T_{s_2s_1}\otimes v_t)\\
T_0\cdot(T_{s_1s_2}\otimes v_t)&=t^{\alpha_2^{\vee}}(T_1\otimes v_t)\\
T_0\cdot(T_{s_2s_1}\otimes v_t)&=t^{\alpha_1^{\vee}}(T_2\otimes v_t)\\
T_0\cdot(T_{s_1s_2s_1}\otimes v_t)&=t^{\varphi^{\vee}}(1\otimes v_t).
\end{align*}

\medskip

\noindent\textbf{Induced representations.} Let $\scH_1$ be the (infinite dimensional) subalgebra of $\scH_W$ generated by $T_1$ and $\CC[P]$. Let $u\in\CC^{\times}$, and let $\CC v_u$ be a $1$-dimensional representation of $\scH_1$ with 
$$
T_1\cdot v_u=-q^{-\frac{1}{2}}v_u,\quad x^{\alpha_1^{\vee}}\cdot v_u=q^{-1}v_u,\quad\textrm{and}\quad x^{\alpha_2^{\vee}}\cdot v_u=q^{\frac{1}{2}}u\,v_u.
$$
Let $(\pi_u^{(1)},V(u))$ be the induced representation of $\scH_W$ with representation space given by $V(u)=\mathrm{Ind}_{\scH_1}^{\scH_W}(\CC v_u)=\scH_W\otimes_{\scH_1}\CC v_u$. The representation space $V(u)$ has basis $\{1\otimes v_u,T_2\otimes v_u,T_{s_1s_2}\otimes v_u\}$, and straightforward computations give
$$
\begin{aligned}
T_1\cdot(1\otimes v_u)&=-q^{-\frac{1}{2}}\,1\otimes v_u\\
T_1\cdot(T_2\otimes v_u)&=T_{s_1s_2}\otimes v_u\\
T_1\cdot(T_{s_1s_2}\otimes v_u)&=T_2\otimes v_u+\fq\,T_{s_1s_2}\otimes v_u
\end{aligned}\quad
\begin{aligned}
T_2\cdot(1\otimes v_u)&=T_2\otimes v_u\\
T_2\cdot(T_2\otimes v_u)&=1\otimes v_u+\fq\,T_2\otimes v_u\\
T_2\cdot(T_{s_1s_2}\otimes v_u)&=-q^{-\frac{1}{2}}\,T_{s_1s_2}\otimes v_u
\end{aligned}
$$
giving the matrices stated in Section~\ref{sect:statement}. One can compute $\pi^{(1)}_u(T_0)$ using the formula $x^{\varphi^{\vee}}=T_0T_1T_2T_1$, or by using \cite[(3.3.6)]{macblue}. The result is given in Section~\ref{sect:statement}.

\medskip

\noindent\textbf{The $1$-dimensional representation.} The representation $\pi^{(2)}$ with representation space $\CC$ has $\pi^{(2)}(T_i)=-q^{-\frac{1}{2}}$ for $i=0,1$ and $2$, and since $x^{\alpha_1^{\vee}}=T_2^{-1}T_0T_2T_1$ and $x^{\alpha_2^{\vee}}=T_1^{-1}T_0T_1T_2$ we have $\pi^{(2)}(x^{\alpha_1^{\vee}})=\pi^{(2)}(x^{\alpha_2^{\vee}})=q^{-1}$.


\subsection*{Acknowledgment}
The first author thanks Donald Cartwright for helpful discussions on related topics over many years. It is a pleasure to present this paper at a conference in honour of his birthday. Indeed both Donald Cartwright and Jean-Phillipe Anker suggested this problem to us, and we thank them both very warmly. The first author also thanks Arun Ram for teaching him about affine Hecke algebras. We also thank E. Opdam for helpful conversations regarding his work on the Plancherel measure of affine Hecke algebras. Finally, thank you to Wolfgang Woess for organising the workshop \textit{Boundaries} in Graz, Austria, June-July 2009. Part of the research for this paper was undertaken in Graz where the first author was supported under the FWF (Austrian Science Fund) project number P19115-N18.

\end{document}